\documentclass{elsarticle}

\usepackage[utf8]{inputenc} 
\usepackage[T1]{fontenc}    
\usepackage{hyperref}       
\usepackage{url}            
\usepackage{booktabs}       
\usepackage{amsfonts}       
\usepackage{nicefrac}       
\usepackage{microtype}      
\usepackage{amsthm,amsmath,amssymb}
\usepackage{algorithm,algpseudocode}
\usepackage{booktabs}
\usepackage{comment}
\usepackage{wrapfig}
\usepackage{enumerate}

\usepackage[skip=2pt]{caption}

\usepackage{pgfplots}
\pgfplotsset{compat=newest}
\usepackage{tikz}
\usetikzlibrary{patterns}
\usetikzlibrary{calc}

\theoremstyle{definition}
\newtheorem{theorem}              {Theorem}
\newtheorem{lemma}      [theorem] {Lemma}
\newtheorem{proposition}[theorem] {Proposition}
\newtheorem{corollary}  [theorem] {Corollary}

\newtheorem{definition} [theorem] {Definition}
\newtheorem{example}    [theorem] {Example}
\newtheorem{remark}     [theorem] {Remark}

\newcommand{\argmax}{\operatornamewithlimits{argmax}}

\newcommand{\adm}[1]{\mathrm{adm}(#1)}

\newcommand{\J}{\mathcal{J}}
\newcommand{\ex}[1]{\mathrm{ex}(#1)}
\newcommand{\coex}[1]{\mathrm{ex}^*(#1)}
\makeatletter
\providecommand{\bigsqcap}{%
  \mathop{%
    \mathpalette\@updown\bigsqcup
  }%
}
\newcommand*{\@updown}[2]{%
  \rotatebox[origin=c]{180}{$\m@th#1#2$}%
}
\makeatother

\newcommand{\lat}{\mathcal{L}}
\newcommand{\ind}{\mathcal{I}}

\newcommand{\base}{\mathcal{B}}
\newcommand{\dep}{\mathcal{D}}

\newcommand{\R}{\mathbb{R}}

\usepackage{color}
\usepackage{CJKutf8}
\usepackage{ifthen}
\newcommand{\COMM}[2]{{
\begin{CJK}{UTF8}{ipxm}
\ifthenelse{\equal{#1}{SN}}{\color{blue}}{
\ifthenelse{\equal{#1}{TM}}{\color{red}}{
\ifthenelse{\equal{#1}{AA}}{\color{cyan}}{
\ifthenelse{\equal{#1}{BB}}{\color{magenta}}}}}
[#1: #2]
\end{CJK}
}}

\begin{document}

\begin{frontmatter}

\title{Rank axiom of modular supermatroids: \\  A connection with directional DR submodular functions}
\author[1]{Takanori Maehara}
\address[1]{RIKEN AIP}
\author[2]{So Nakashima}
\address[2]{the University of Tokyo}

\begin{abstract}
A matroid has been one of the most important combinatorial structures since it was introduced by Whitney as an abstraction of linear independence.
As an important property of a matroid, it can be characterized by several different (but equivalent) axioms, such as the augmentation, the base exchange, or the rank axiom.

A supermatroid is a generalization of a matroid defined on lattices.
Here, the central question is whether a supermatroid can be characterized by several equivalent axioms similar to a matroid.
Barnabei, Nicoletti, and Pezzoli characterized supermatroids on distributive lattices, and Fujishige, Koshevoy, and Sano generalized the results for  cg-matroids (supermatroids on lower locally distributive lattices).

In this study, we focus on modular lattices, which are an important superclass of distributive lattices, and provide equivalent characterizations of supermatroids on modular lattices. 
We characterize supermatroids on modular lattices using the rank axiom in which the rank function is a directional DR-submodular function, which is a generalization of a submodular function introduced by the authors.
Using a characterization based on rank functions, we further prove the strong exchange property of a supermatroid, which has application in optimization.

We also reveal the relation between the axioms of a supermatroid on lower semimodular lattices, which is a common superclass of a lower locally distributive lattice and a modular lattice.
\end{abstract}

\end{frontmatter}

\tableofcontents

\section{Introduction}

\subsection{Background and Motivation}

A \emph{matroid} is a combinatorial structure first introduced by Whitney~\cite{whitney1935abstract} as an abstraction of linear independence.
A matroid is now regarded as one of the most important combinatorial structures and has been studied in various fields such as combinatorics, geometry, and optimization.

Several generalizations of matroid have been proposed.
Here, we focus on a \emph{supermatroid} on lattices as introduced by Dunstan, Ingleton, and Welsh~\cite{dunstan1972supermatroids}.
Let $\lat$ be a finite lattice.
A supermatroid $\ind$ on $\lat$ is an ideal of $\lat$ such that for all $X \in \lat$, the maximal elements in $\ind^X = \{ I \in \ind \mid I \le X \}$ have the same height (i.e., distance from the bottom of $\lat$).
If $\lat$ is a boolean lattice, the supermatroids coincide with the classical matroids.

The classical matroids are characterized by several different (but equivalent) axioms, such as the augmentation, the base exchange, and the rank axioms.
Here, the rank function $r$ of a (super)matroid $\ind$ is defined by
\begin{align}
    r(X) = \max \{ | I | : I \in \ind \}.
\end{align}
Thus, it is natural to ask whether supermatroids have similar equivalent axioms (see \cite{wild2008weakly}).
Barnabei, Nicoletti, and Pezzoli~\cite{barnabei1998matroids} answered this question for supermatroids on distributive lattices, and Fujishige, Koshevoy, and Sano~\cite{fujishige2007matroids} and Sano~\cite{sano2008rank} generalized the results for supermatroids on convex geometries (i.e., lattices of convex sets), where the convex geometries are equivalent to the lower locally distributive lattices.

Our aim is to extend this line of study.
In particular, we are interested in supermatroids on \emph{modular lattices} (also called \emph{modular supermatroids}~\cite{li2014base}).
Modular lattices are a superclass of distributive lattices and are incomparable with lower locally distributive lattices (see Figure~\ref{fig:lattices}).
Modular lattices naturally occur in algebra.
For example, the set of subspaces of a vector space forms a modular lattice.
In addition, the set of normal subgroups of a group form a modular lattice.
Modular lattices are also important in optimization and machine learning.
The authors formulated a subspace selection problem, including principal component analysis, as a maximization problem over modular lattices, and showed that if the objective function satisfies a submodular-like property, called \emph{directional DR-submodularity}, the height and the knapsack constrained problems can be solved within provable approximation factors~\cite{nakashima2018subspace}.

\begin{figure}[tb]
\centering
\begin{tikzpicture}
\node[draw] at (0,0) (d) {Distributive};
\node[draw] at (3,1) (m) {Modular};
\node[draw] at (-3,1) (lld) {Lower Locally Distributive};
\node[draw] at (0,2) (llm) {Lower Locally Modular};
\node[draw] at (0,3) (lsm) {Lower Semimodular};
\draw[->] (d)--(m);
\draw[->] (d)--(lld);
\draw[->] (m)--(llm);
\draw[->] (lld)--(llm);
\draw[->] (llm)--(lsm);
\end{tikzpicture}
\caption{Relation of the underlying lattices. Upper one is a super-class of lower one.}
\label{fig:lattices}
\end{figure}
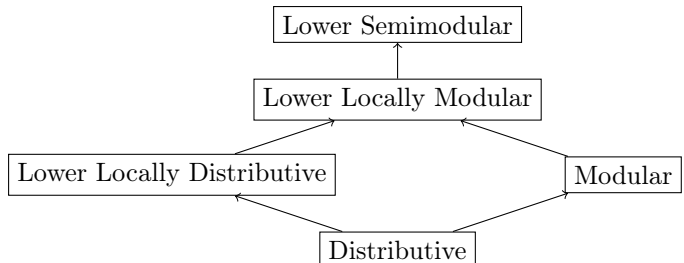

\subsection{Our Contributions}

The contributions of this study are two-fold.
First, we provide several equivalent axioms for supermatroids on modular lattices.
Second, we show that a further generalization for a supermatroid on lower-locally modular lattices, which is a common superclass of lower-locally distributive lattices (studied in \cite{fujishige2007matroids,sano2008rank}) and modular lattices (see Figure~\ref{fig:lattices}) is not straight-forward.

\paragraph{Supermatroids on Modular Lattices}

We defined supermatroids via the height axiom.
Here, we show that supermatroids on \emph{modular lattices} can be equivalently characterized by 
the independence (augmentation) axiom (Theorem~\ref{thm:modular-ind-height}),
the middle base axiom (Theorem~\ref{thm:mid}),
the rank axiom (Theorem~\ref{thm:ind-rank-mod}), 
and the dependence axiom (Theorem~\ref{thm:mod-dep-ind}).
Here, the most interesting result is the rank axiom.

In the case of distributive lattices, the rank function must be a submodular function~\cite{barnabei1998matroids}, 
i.e., the rank function $r \colon \lat \to \mathbb{R}$ satisfies the following \emph{lattice-submodularity}:
\begin{align}
    r(X \lor Y) + r(X \land Y) \le r(X) + r(Y).
\end{align}
More precisely, it satisfies the following \emph{DR-submodularity}~\cite{soma2016maximizing,gottschalk2015submodular}:
\begin{align}
    r(Y \lor b) - r(Y) \le r(X \lor a) - r(X),
\end{align}
where $X, Y \in \lat$; $a \in \adm{X}$; and $b \in \adm{Y}$ such that $X \le Y$ and $a \le b$.
Here, $a \in \adm{X}$ intuitively indicates that $a$ is safely added to $X$ (see Section~\ref{subsec:preliminaries-lattice} for a precise definition).
We can see that the DR-submodularity implies the lattice-submodularity when the lattice is distributive (see Remark~\ref{rem:distributive-DR}).
However, this is not the case in modular lattice, where the rank function does not satisfy the DR-submodular inequality (Proposition~\ref{prop:DR-submodular-is-needed}). 
Thus, we relax the DR-submodularity as follows:

A function $r \colon \lat \to \mathbb{R}$ is a \emph{downward DR-submodular function} if it satisfies the following inequality:
\begin{align}
    r(Y \lor b) - r(Y) \le \max_{a \le a'} \min_{a \colon Y \lor a' = Y \lor b} r(X \lor a) - r(X),
\end{align}
where $X, Y \in \lat$; $a \in \adm{X}$; and $b \in \adm{Y}$ such that $X \le Y$ and $a \le b$.
This function class was introduced by Nakashima and Maehara~\cite{nakashima2018subspace} as a generalization of the DR-submodular function to characterize tractable subspace selection problems (see Section~\ref{subsec:directional-DR}).
Our rank axiom requires the rank function to be a downward DR-submodular function.

Our rank axiom allows us to derive the strong exchange property of supermatroids on modular lattices (Theorem~\ref{thm:strongexchange}), which has applications in optimization problems (see Section~\ref{subsubsec:valuated-supermatroid} and Section~\ref{subsubsec:DR-submodular-opt}).

\paragraph{Supermatroids on Lower Locally Modular Lattices}

It is a natural question whether the relation between the axioms remains true on a class of lattices that is more general than the modular lattices.
Here, we consider \emph{lower-locally modular lattices}, which is a common generalization of modular lattices and lower-locally distributive lattices.

We show that the independence axiom characterizes the matroids; however, all other axioms do not (Section~\ref{sec:llm}).

\subsection{Related Studies}

\paragraph{Matroids on Lattices}

Most studies on supermatroids have focused on supermatroids on a distributive lattice~\cite{barnabei1993symmetric, barnabei1998matroids, li2004global}, which are also called \emph{poset matroids}.
A further generalization was been conducted.
A matroid is a generalization of an arrangement of geometric objects.
In this direction, a\emph{cg-matroid}~\cite{fujishige2007matroids} has also been studied.
Cg-matroids are equivalent to supermatroids over lower locally distributive lattices~\cite{edelman1980meet}. 
Fujishige, Koshevoy, and Sano~\cite{fujishige2007matroids} studied supermatroids on lower locally distributive lattices, and proved the equivalence of some axioms.
Sano~\cite{sano2008rank} characterized supermatroids on lower locally distributive lattices in terms of their rank function.
Another study was been conducted on \emph{modular lattices}. 
Li and Liu~\cite{li2014base} showed the base axiom for supermatroids on modular lattices.

One of the purposes of this study is to complete the characterization of supermatroids on modular lattices.
Another is to unify the theory of lower locally distributive and modular lattices.
We consider \emph{lower locally modular lattices}, which contain both lower locally distributive lattices and modular lattices as proper subclasses.

\paragraph{DR-Submodular Function}

In the machine-learning community, submodular functions (on the set lattices) have been widely studied.
Soma and Yoshida~\cite{soma2016maximizing} introduced a class of functions on an integer lattice ($\mathbb{Z}^n$), called \emph{DR-submodular functions}, as a generalization of the submodular functions on set lattices.
The authors generalized this concept for a \emph{modular lattice} and derived an approximation algorithm~\cite{nakashima2018subspace} for maximization problems of the functions.

\section{Preliminaries}

\subsection{Lattice}
\label{subsec:preliminaries-lattice}
Let $(P, \le)$ be a poset.
A subset $I \subseteq P$ is an \emph{ideal} if $p \in I$ for any $p,p' \in P$ with $p \le p'$ and $p' \in I$.
A subset $F \subseteq P$ is a \emph{filter} if the complementary set $I = P \setminus F$ is an ideal.
A subset $I \subseteq P$ is proper if $I \neq P$.
By $P^*$, we denote the order-reversal of $P$.

A \emph{lattice} $(\lat, \le)$ is a partially ordered set closed under the least upper bound $X \lor Y := \inf \{ Z \mid Z \ge X, Z \ge Y \}$ and the greatest lower bound $X \land Y := \sup \{ Z \mid Z \le X, Z \le Y \}$.
We state that $X$ \emph{covers} $Y$ if $X \le Y$ and there is no $Z \in \lat$ with $X \lneq Z \lneq Y$, and denote $X \prec Y$ if $X$ covers $Y$.
For $X_1, X_2, Y \in \mathcal{L}$, we denote by $X_1 = X_2 \bmod Y$ if $X_1 \lor Y = X_2 \lor Y$.
An element $a \in \lat$ is \emph{join-irreducible} if $a = X \lor Y$ implies $X = a$ or $Y = a$. 
To avoid confusion, we use upper-case letters for general lattice elements and lower-case letters for join-irreducible elements.
The set of join-irreducible elements is denoted by $\J$. 
A lattice $\lat$ is \emph{atomic} if any pair of $\J$ is incomparable. 
A join-irreducible element $a \in \J$ is \emph{admissible} with respect to $X$ if $X \land a \prec a$. 
We denote by $\adm{X}$ the set of admissible elements with respect to $X$. 

A lattice $\mathcal{L}$ is \emph{lower-semimodular (LSM)} if $X \prec X \lor Y$ implies $X \land Y \prec Y$, and is \emph{upper-semimodular (LUM)} if $X \land Y \prec X$ implies $X \prec X \lor Y$.
A lattice $\mathcal{L}$ is \emph{modular} if it is both lower- and upper-semimodular.
A lattice $\mathcal{L}$ is \emph{distributive} if for any $X, Y, Z \in \mathcal{L}$, the following identity holds:
\begin{align}
    X \land (Y \lor Z) = (X \land Y) \lor (X \land Z).
\end{align}
Note that any distributive lattice is also a modular lattice.

For $X \in \mathcal{L}$, let $X^- = \bigwedge \left\{ Y \in \mathcal{L} \mid Y \prec X \right\}$.
The lattice is said to be a \emph{lower locally modular lattice} if for each $X \in \mathcal{L}$, the interval $[X^-, X]$ is a modular lattice, and \emph{lower locally distributive} if $[X^-, X]$ is a distributive lattice.
Any lower locally distributive lattice has a geometric realization called a \emph{convex geometry}~\cite{edelman1980meet}.
The relations of the lattices are summarized in Figure~\ref{fig:lattices}.

For $X, Y \in \lat$ with $X \le Y$, a \emph{chain} is a sequence of the form $X = X_0 \le X_1 \le \dots \le X_m = Y$.
In a lower- or upper-semimodular lattice, for any two elements $X, Y$ with $X \le Y$, the length of the maximal chain, $X = X_0 \prec X_1 \prec \dots \prec X_m = Y$, is the same~\cite{birkhoff1940lattice}, which is called the \emph{Jordan--Dedekind property}.
In lower-semimodular lattices, the height $|X|$ is the length of the maximal chain from $\bot$ to $X$.
The height of a lattice $\lat$ is $|\top|$.

A lattice is a \emph{finite length} if the maximum length of the chains is bounded. 
This study only considers the lattices of a finite length.
A finite length lattice is \emph{bounded}, i.e., it has the top $\top$, which is greater than any other elements, and the bottom $\bot$, which is smaller than any other elements.
For $X \in \lat$, the height $|X|$ is the length of the maximal ascending chain from the bottom to $X$.

A fundamental example of a modular lattice is a \emph{vector lattice}.
A vector lattice $\lat(\mathbb{R}^n)$ consists of all linear subspaces of $\mathbb{R}^n$ equipped with the inclusion order $\le := \subseteq$.
For $X,Y \in \lat(\mathbb{R}^2)$, we have $X \land Y = X \cap Y$, and $X \lor Y$ equals the direct sum of $X$ and $Y$.
We can easily see that $\lat(\mathbb{R}^n)$ is an atomic modular lattice and that join-irreducible elements are of the form $\mathrm{span}(v)$ for a certain $v \in \mathbb{R}^n \setminus \{0\}$.

\subsection{Directional DR-submodular Functions}
\label{subsec:directional-DR}

\begin{definition}[directional DR-submodularity~\cite{nakashima2018subspace}]
\label{def:downwarddr}
Let $\mathcal{L}$ be a lattice.
A function $f \colon \mathcal{L} \to \mathbb{R}$ is \emph{downward DR-submodular} if for all $X, Y \in \mathcal{L}$ with $X \le Y$ and $b \in \adm{Y}$, there exists $b' = b \bmod Y$ such that, for all $a \in \adm{X}$ with $a \le b'$, the following holds:
\begin{align}
\label{eq:downwarddr}
f(Y \lor b) - f(Y) \le   f(X \lor a) - f(X).
\end{align}

A function $f$ is \emph{upward DR-submodular} if $f$ is downward DR-submodular on $\lat^*$.
If a function $f$ is both downward DR-submodular and upward DR-submodular, we state that $f$ is \emph{bidirectional DR-submodular}.
\end{definition}

\begin{remark}
The definition of the upward DR-submodularity is different from the original~\cite{nakashima2018subspace}.
We introduced a weaker definition to investigate the dualistic nature of the modular supermatroids.
See Section~\ref{subsec:modulardual}.
\end{remark}

\section{Supermatroids on Modular Lattices}
\label{sec:modular}

In this section, we reveal the relationship between directional DR-submodularity and modular supermatroids:
Directional DR-submodularity naturally arises as a property of the rank function of a modular supermatroid.
We also introduce some new characterizations of modular supermatroids.
In this section, $\lat$ is a modular lattice and $\ind \subseteq \lat$ if there are no futher specifications.
All proofs are given in Section~\ref{sec:proof}.

\subsection{Height Axiom}

Let $P = (P, \le)$ be a poset.
A subset $\ind \subseteq P$ is a \emph{supermatroid} if it satisfies the following \emph{height axiom}.
\begin{definition}[Height Axiom] 
\label{def:height}
\begin{description}
    \item[(H1)] $\ind$ is a non-empty ideal;
    \item[(H2)] For all $X \in P$, the maximal elements in the intersection of $\ind$ and $\ind^X :=\{X' \in P \mid X' \le X \}$ have the same height.
\end{description}
\end{definition}

The \emph{rank} of a modular supermatroid $\ind$ is the height of the maximal elments of $\ind$, which is uniquely determined by (H2).
A supermatroid over a modular lattice is called a \emph{modular supermatroid}.
A supermatroid over a distributive lattice is called a \emph{distributive supermatroid} or a \emph{poset matroid}.

\subsection{Independence Axiom}

Modular supermatroid $\ind$ is alternatively defined by the following augumentation axiom:
Consider the following condition on $\ind$.
\begin{definition}[Independence Axiom]
\label{def:independence}
\begin{description}
    \item[(I1)] $\ind$ is a non-empty ideal;
    \item[(I2)] For all $I_1, I_2 \in \ind$ with $|I_1| < |I_2|$, there exists $J \in \ind$ such that $I_1 < J \le I_1 \lor I_2$.
\end{description}
\end{definition}

\begin{theorem}[Independence Axiom $\Leftrightarrow$ Height Axiom]
\label{thm:modular-ind-height}
\begin{itemize}
    \item[(1)] Any supermatroid $\ind$ satisfies (I1) and (I2). 
    \item[(2)] Any $\ind$ satisfying (I1) and (I2) is a supermatroid.
\end{itemize}
\end{theorem}

We can weaken the independence axiom.
The ``localized'' augmentation (Figure~\ref{fig:local-augmentation}) is sufficiently strong to characterize the modular supematroids.
\begin{description}
    \item[(I2l)]  For all $I_1, I_2 \in \ind$ with $|I_1| + 1 = |I_2|$ and $I_2 \prec I_1 \lor I_2$, there exists $Z \le I_1 \lor I_2$ such that $Z \in \ind$ and $I_1 < Z$.
\end{description}

\begin{figure}[tb]
\centering
\begin{tikzpicture}[scale=0.8]
\node[draw, circle, label=west:{$I_1 \lor I_2$}] at (2,2) (I1I2) {};
\node[draw, circle, label=west:{$I_2$}] at (3,1) (I2) {};
\node[draw, circle, label=west:{$J$}] at (1,1) (J) {};
\node[draw, circle, label=west:{$I_1$}] at (0,0) (I1) {};
\draw[-] (I1)--(J);
\draw[-] (J)--(I1I2);
\draw[-] (I2)--(I1I2);
\end{tikzpicture}
\caption{Local Augmentation. Each line shows a cover relation.}
\label{fig:local-augmentation}
\end{figure}
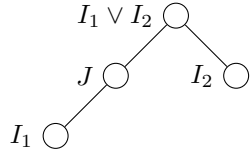

\begin{proposition}
\label{prop:localaug}
Under (I1), the two conditions (I2) and (I2l) are equivalent.
\end{proposition}

\subsection{Rank Axiom}

We characterize modular supermatroids in terms of their rank function.
Let $r \colon \lat \to \mathbb{R}$.
Consider the following condition on $r$.
\begin{definition}[Rank Axiom]
\label{def:rank}
\begin{description}
    \item[(R1)] $r(\perp) = 0$;
    \item[(R2)] $r(X) - r(X') \in \{0,1\}$ for all $X \succ X'$;
    \item[(R3)] $r$ is a downward DR-submodular function.
\end{description}
\end{definition}

\begin{theorem}[Rank Axiom $\Leftrightarrow$ Independence Axiom]  
\label{thm:ind-rank-mod}
\begin{itemize}
    \item[(1)] Let $\ind$ be a modular supermatroid. Then, $r(X) := \min \{|I| \mid I \le X, I \in \ind\}$ satisfies the rank axiom. 
    \item[(2)] Conversely, let $r$ be a function over a modular lattice satisfying the rank axiom. Then, $\ind := \{X \in \lat \mid r(X) = |X|\}$ is a modular supermatroid.
\end{itemize}
\end{theorem}

The condition (R3) appears to lack symmetry.
Thus, it would be natural to consider the upward and the bidirectional versions as follows:
\begin{description}
    \item[(R3u)] $r$ is a  upward DR-submodular function;
    \item[(R3b)] $r$ is a bidirectional DR-submodular function.
\end{description}

\begin{proposition}
\label{prop:rank-upward}
Let $\lat$ be a lattice and $r \colon \lat \to \mathbb{R}$ be a function that satisfies (R1) and (R2). 
Then, $r$ satisfies (R3) if and only if $r$ satisfies (R3u).
\end{proposition}

In particular, (R1), (R2), and (R3b) are other axioms of a rank function.

\begin{remark}
\label{rem:distributive-DR}
Theorem~\ref{thm:ind-rank-mod} implies that the rank function of a supermatroid on a distributive lattice is a DR-submodular function. 
The downward DR-submodular function coincides with the DR-submodular function~\cite{nakashima2018subspace}.
\end{remark}
 
\begin{remark}
We remark that (R3) \emph{cannot} be replaced with ``$r$ is a lattice-submodular function.''

\begin{proposition}[The lattice submodularity is insufficient]
\label{prop:DR-submodular-is-needed}
There exists a modular lattice $\lat$ and a supermatroid $\ind$ on $\lat$ such that the rank function $r(X) = \max \{ |I| \mid I \in \ind, I \le X \}$ is not lattice-submodular.
\end{proposition}
\begin{proof}
Let $\lat = \{\bot, a, b, c, \top\}$ be a diamond lattice and $\ind = \{ \bot, a \} \subseteq \lat$ as shown in Figure~\ref{fig:diamond}, where the elements in $\ind$ are indicated by the black points.
The subset $\ind$ is a supermatroid. 
However, the rank function does not satisfy the lattice submodularity because $r(b) + r(c) = 0 \not \ge r(b \lor c) + r(b \land c) = 1$.
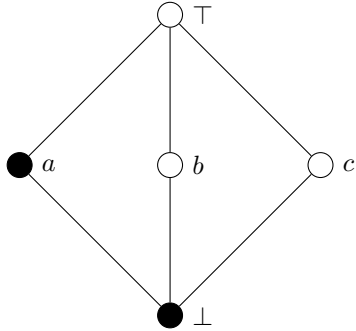
\begin{figure}[tb]
    \centering
    \begin{tikzpicture}[scale = 2]
    \node[draw,circle,fill,label=east:{$\bot$}] at (0,0) (o) {};
    \node[draw,circle,fill,label=east:{$a$}] at (-1,1) (a) {};
    \node[draw,circle,label=east:{$b$}] at (0,1) (b) {};
    \node[draw,circle,label=east:{$c$}] at (1,1) (c) {};
    \node[draw,circle,label=east:{$\top$}] at (0,2) (t) {};
    \draw[-] (o)->(a);
    \draw[-] (o)->(b);
    \draw[-] (o)->(c);
    \draw[-] (a)->(t);
    \draw[-] (b)->(t);
    \draw[-] (c)->(t);
    \end{tikzpicture}
    \caption{A supermatroid $\ind = \{\bot, a\}$ on a modular lattice (the diamond lattice) whose rank function does not satisfy the lattice-submodularity.}
    \label{fig:diamond}
\end{figure}
\end{proof}

Note that the rank function of a cg-matroid satisfies the submodularity~\cite{sano2008rank}.
Thus, the directional DR-submodularity is required to handle the modularity of the lattice.
\end{remark}
This result indicates the naturality of our directional DR-submodularities.
Directional DR-submodularities are originally introduced to characterize the maximization problems solved by a greedy algorithm over the lattices~\cite{nakashima2018subspace}.
Surprisingly, directional DR-submodularties are indespensable concepts in a different subject, i.e.,  modular supermatroids.

\subsection{Middle Axiom (Base Axiom)}
In this section, we derive a middle axiom for base, maximal elemets of $\ind$, of modular supermatroids.
Consider the following conditions on $\base \subseteq \lat$:
\begin{definition}[Middle Axiom] \ 
\begin{description}
\item[(B1)] All elements in $\mathcal{B}$ are pairwise incomparable, i.e., $B_1, B_2 \in \mathcal{B}$, $B_1 \le B_2$ implies $B_1 = B_2$.
\item[(B2)] For all $B_1, B_2 \in \mathcal{B}$, $X, Y \in \mathcal{L}$ such that $X \le B_1$ and $B_2 \le Y$, there exists $B \in \mathcal{B}$ such that $X \le B \le Y$.
\end{description}
\end{definition}

\begin{lemma}
\label{lem:base-sameheight}
Let $\mathcal{L}$ be a lower semimodular lattice.
Then, all elements in $\mathcal{B}$ have the same height.
\end{lemma}
\begin{theorem}[Middle Axiom $\Leftrightarrow$ Independence Axiom]
\label{thm:mid}
\begin{itemize}
    \item[(1)] Let $\base$ be the family of  maximal elements of a modular supermatroid $\ind$. Then, $\base$ satisfies (B1), (B2), and (B3).
    \item[(2)] Conversely, let $\base \subseteq \lat$ satisfy (B1), (B2), and (B3). Let $\ind := \{ X \in \lat \mid \exists B \in \base \; X \le B\}$.
    Then, $\ind$ satisfies (I1) and (I2) if the lattice is lower semimodular.
\end{itemize}
\end{theorem}

\subsection{Dependence Axiom}
In this section, we introduce an axiom for $\dep := \lat \setminus \ind$.
This is a substitute for the standard axiom for \emph{circuits} (minimal elements of $\dep$).
Introducing an axiom for the circuits of the modular supermatroids is an unsolved problem raised in \cite{li2014base}.
Here, we only consider $\dep$ because the axiom for the circuits of modular supermatroids is not simple and has almost the same form as that of $\dep$.
Consider the following conditions for $\dep \subseteq \lat$.
\begin{definition}[Dependence Axiom]
\begin{description}
\item[(D1)] $\mathcal{D}$ and is a proper filter.
\item[(D2)] For all $D_1, D_2 \in \dep$ and $Z \prec D_1 \lor D_2$ such that $D_i \prec D_1 \lor D_2$, $Z \neq D_i$ ($i = 1,2$), and $D_1 \land D_2 \not \le Z$, we have at least one of (1) $Z \in \dep$, (2) $D_1 \land D_2 \in \dep$, or (3) there exists $I \not \in \dep$ such that $D_1 \land D_2 \prec I \prec D_1 \lor D_2$.
\item[(D3)] Let $X, W \in \lat$ with $W \prec X \lor W$.
If there uniquely exists $Y$ such that $Y \in \dep$ and $X \prec Y \prec X \lor W$, then we have $X \in \dep$ or $W \in \dep$.
\end{description}
\end{definition}
We note that (D2) is called the elimination axiom and (D3) the replacement axiom in~\cite{barnabei1998matroids}.

\begin{theorem}[Dependence Axiom $\Leftrightarrow$ Independence Axiom]
\label{thm:mod-dep-ind}
\begin{itemize}
    \item[(1)] Let $\dep$ be a complement of independent sets of a modular supermatroid. Then, $\dep$ satisfies (D1), (D2), and (D3).
    \item[(2)] Conversely, let $\dep \subset \lat$ satisfying (D1), (D2), and (D3). Then, $\ind := \lat \setminus \dep$ satisfies (I1) and (I2).
\end{itemize}
\end{theorem}

\subsection{Strong Basis Exchange Theorem}

We prove the strong exchange property of modular supermatroids.

\begin{remark}
As a difficulty in generalizing the strong exchange property to modular supermatroids, we cannot uniquely define the ``minus'' operation $X - a := X \setminus \{a\}$ for a set $X$ and an element $a \in X$ for modular lattices, which is used in the statement of the usual strong exchange property.
More precisely, there might be multiple values of $\mathring X \prec X$ such that $\mathring X \lor a = X$ for $\mathring X \prec X$ and $a \in \adm{\mathring{X}}$.
Let us consider a vector lattice $\lat(\mathbb{R}^2)$.
Let $a = \mathrm{span}((1,0)^\top)$ and $X = \top$.
In this setting, all one-dimensional spaces of the form $\mathring X=\mathrm{span}(v)$ for $v \neq (1,0)^\top$ satisfy $X := \top = \mathring X \lor a$.

We overcome this difficulty by formulating a strong exchange property through Theorem~\ref{thm:strongexchange}.
The formulation is natural in that it implies a usual strong exchange property for Boolean lattices   (Remark~\ref{rem:strong-exchange-set}) and the base axiom (Proposition~\ref{prop:strongex-to-base}).
In addition, it can be applied to an optimization (Sections~\ref{subsubsec:valuated-supermatroid}~and ~\ref{subsubsec:DR-submodular-opt}).
\end{remark}

\begin{theorem}[Strong Exchange Property]
\label{thm:strongexchange}
Let $X, Y \in \base$ and $\mathring{X} \prec X$ such that  $\mathring{X} \ge X \land Y$.
Then, a $w \in \adm{\mathring{X}}$ with $\mathring{X} \lor w = X$ that satisfies the following exists.
For all $\underline w \le w$ with $\underline w \in \adm{Y}$, there also exists $Y' \in \base$ such that $(X \land Y) \lor \underline w \le Y' \prec Y \lor \underline w$.
In addition, for all  $y \in \adm{Y'}$ with $Y' \lor y = Y \lor \underline w$, there exist join-irreducible elements $\underline y \le y$ and $v$ such that $\underline y = v \bmod Y \land Y'$ and $\mathring{X} \lor \underline v \in \base$ for all $\underline v \in \adm{\mathring{X}}$ with $\underline v \le v$.
Here, $\underline v \in \adm{\mathring{X}}$ with $\underline v \le v$ is non-empty.
\end{theorem}

Figure~\ref{fig:symmetric-exchange} shows the relationship among the elements appearing in this theorem.

\begin{figure}[t]
        \begin{minipage}{0.5\hsize}
        \centering
            \begin{tikzpicture}[scale = 1]
                \node[draw,circle, fill, label=north:{$X$}, inner sep=0pt,minimum size=2mm]  at (0,0) (X) {};
                \node[draw,circle,label=east:{$\mathring{X}$}, inner sep=0pt,minimum size=2mm] at (1,-1) (intX) {};
                \node[draw,circle, fill,label=north:{$\mathring{X} \lor \underline v$}, inner sep=0pt,minimum size=2mm] at (2,0) (xv) {};
                \node[draw,circle,fill,label=east:{$Y$}, inner sep=0pt,minimum size=2mm] at (3,0) (Y) {};
                \node[draw,circle,label=north:{$Y \lor \underline w$}, inner sep=0pt,minimum size=2mm] at (4,1) (yw) {};
                \node[draw,circle, fill, label=east:{$Y' $}, inner sep=0pt,minimum size=2mm] at (5,0) (yprime) {};
                \draw[-] (X)->(intX) node[pos = 0.35, right] {$+ w$};
                \draw[-] (intX)->(xv);
                \draw[-] (Y)->(yw);
                \draw[-] (yw)->(yprime) node[pos = 0.35, right] {$+y$};
            \end{tikzpicture}
    \end{minipage}
    \begin{minipage}{0.5\hsize}
        \centering
        \begin{tikzpicture}[scale = 1]        
        \node[draw,circle, label=east:{$w$}, inner sep=0pt,minimum size=2mm]  at (0,1) (w) {};
        \node[draw,circle,  label=east:{$\underline w$}, inner sep=0pt,minimum size=2mm]  at (0, -1) (uw) {};
        \node[draw,circle, label=west:{$y$}, inner sep=0pt,minimum size=2mm]  at (2,1) (y) {};
        \node[draw,circle, label=west:{$\underline y$}, inner sep=0pt,minimum size=2mm]  at (2,0) (uy) {};
        \node[draw,circle, label=east:{$v$}, inner sep=0pt,minimum size=2mm]  at (3,0) (v) {};
        \node[draw,circle,  label=east:{$\underline v$}, inner sep=0pt,minimum size=2mm]  at (3,-1) (uv) {};
        \draw[-] (w)->(uw) ;
        \draw[-] (y)->(uy) ;
        \draw[dotted] (uy)->(v);
        \draw[-] (v)->(uv);
        \end{tikzpicture}
    \end{minipage}
    \caption{Relationship of lattice elements appearing in Theorem~\ref{thm:strongexchange}. Base elements are indicated by black dots.}
    \label{fig:symmetric-exchange}
\end{figure}
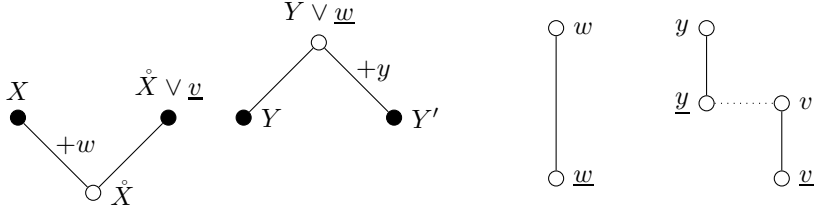

We can use a strong exchange property as the base Axiom:
\begin{proposition}
\label{prop:strongex-to-base}
The strong exchange property implies (B3).
\end{proposition}

\begin{remark}
In the proof of the strong exchange property, the directional DR-submodularity of the rank function plays an important role (Lemma~\ref{lem:strengthend-downward-DR} and its application to the proof of Theorem~\ref{thm:strongexchange}).
This is an application of the rank axiom and directional DR-submodularity.
\end{remark}

\begin{remark}
\label{rem:strong-exchange-set}
In a Boolean lattice, Theorem~\ref{thm:strongexchange} implies the usual strong exchange property.
In Boolean lattice, $w$ is uniquely determined as $X \setminus \mathring X$.
Because all join-irreducible elements are incomparable in the Boolean lattices, $w = \underline w$.
Furthermore, because $Y' \lor v = Y \lor w$ implies $v = w$ in a Boolean lattice, we have two bases $X - w + v$ and $Y + w - v$.
Here, $X + v := X \cup \{v\}$.
Using a similar argument, we can prove that Theorem~\ref{thm:strongexchange} implies a strong exchange property of the distributive lattices~\cite{barnabei1998matroids}.
\end{remark}

For atomic modular lattices, we can simplify the statement of the strong exchange property.

\begin{corollary}
\label{cor:strongexchange}
Let $X, Y \in \base$ and $\mathring{X} \prec X$ such that  $\mathring{X} \ge X \land Y$.
There then exists $w$ with $\mathring{X} \lor w = X$, $Y'$ with $w \lor (X \land Y) \le Y' \prec Y \lor w$, and $v \in Y$ with $Y' \lor v = Y \lor w$ such that
$\mathring{X} \lor v$ and $Y'$ are bases.
Here, $w$ and $v$ are join-irreducible elements.
\end{corollary}
\begin{proof}
Because all join-irreducible elements are incomparable, 
$w = \underline w$ and $y = \underline y$.
The latter and $\underline y = v \bmod Y \land Y'$ imply that $Y' \lor v = y$.
\end{proof}

In the following sections, we provide applications of strong exchange property for mathematical optimization.

\subsubsection{Application 1: Valuated Supermatroid}
\label{subsubsec:valuated-supermatroid}
Valuated matroids~\cite{dress1992valuated} are a quantitative generalization of a matroid (on the Boolean lattices) defined by extending the strong exchange property of the matroids. 
Based on our strong exchange property on the supermatroids on atomic modular lattices, we can generalize the valuated matroids to valuated supermatroids on atomic modular lattices.

Let $\lat$ be a modular lattice, and let $\lat_k = \{X \in \lat \mid |X| = k\}$ for integer $k$.  
A function $\omega \colon \lat \rightarrow \mathbb{R}$ is a \emph{valuated supermatroid} if it satisfies the following:
For all $X,Y \in \lat_k$ and $\mathring X \prec X$ with $X \land Y \le \mathring X$, 
there exists $w \in \adm{\mathring{X}}$ with $\mathring{X} \lor w = X$ that satisfies the following.
For all $\underline w \le w$ with $\underline w \in \adm{Y}$, there exists $Y' \in \lat_k$ such that $(X \land Y) \lor \underline w \le Y' \prec Y \lor \underline w$.
In addition, for all  $y \in \adm{Y'}$ with $Y' \lor y = Y \lor \underline w$, there exist join-irreducible elements $\underline y \le y$ and $v$ such that $\underline y = v \bmod Y \land Y'$ and
\begin{align}
    \label{eq:def-valuated-matroid}
    \omega(X) + \omega(Y) \le  \omega(\mathring X \lor \underline v) + \omega(Y'),
\end{align}
for all $\underline v \in \adm{\mathring{X}}$ with $\underline v \le v$.
We note that $\underline v \in \adm{\mathring{X}}$ with $\underline v \le v$ is non-empty.

\begin{figure}
\begin{algorithm}[H]
\caption{Greedy algorithm for maximization problem of valuated supermatroid.}
\label{alg:greedy_valuated}
\begin{algorithmic}[1]
\State{$X = \bot$}
\For{$i = 1, \ldots, k$}
\State{Let $a_i \in \displaystyle \argmax_{a \in \adm{X}} \omega(X \lor a)$}
\State{$X \leftarrow X \lor a_i$}
\EndFor
\State{\Return $X$}
\end{algorithmic}
\end{algorithm}
\end{figure}

\begin{theorem}
\label{thm:greedy-valuated}
Let $\omega$ be a valuated supermatroid on an atomic modular lattice $\lat$.
Algorithm \ref{alg:greedy_valuated} finds a maximizer of $\omega$.
\end{theorem}

A rank function of a supermatroid on an atomic modular lattice is a valuated supermatroid.
A non-trivial example of this is the objective function of principal component analysis with a single data point:
\begin{proposition}
\label{prop:pca-is-valuated-supermatroid}
Let $ \omega_{\mathrm{PCA}} \colon \lat(\R^n) \to \mathbb{R}$ be
\begin{align}
    \omega_{\mathrm{PCA}}(X) = \| \Pi_X d \|^2,
\end{align}
where  $\Pi_X$ is the projection into subspace $X$, and  $d$ is a vector in $\mathbb{R}^n$.
Then, $ \omega_{\mathrm{PCA}}$ is a valuated modular supermatroid.
\end{proposition}
We conjecture that the objective function of principal component analysis is a valuated supermatroid even if there is more than a single data point.
If this conjecture is true, we can explain the fact that the greedy algorithm finds the optimal solution of $w_{\mathrm{PCA}}$ from the perspective of a matroid.

\subsubsection{Application 2: DR-submodular maximization on supermatroid}
\label{subsubsec:DR-submodular-opt}

The authors considered a maximization problem of a monotone downward DR-submodular function over a modular lattice $\lat$ under the height  and the knapsack constraints~\cite{nakashima2018subspace}.
We proved that the greedy algorithm has approximation guarantees by generalizing the standard results for Boolean lattices.
It seems natural to consider a modular supermatroid constraint, i.e., 
\begin{align}
\label{eq:problem}
    \begin{array}{llll}
         \text{maximize} & f(X) & \text{subject to} & X \in \ind,
    \end{array}
\end{align}
where $f \colon \lat \to \mathbb{R}$ is a downward DR-submodular function and $\ind \subseteq \lat$ is a modular supermatroid. 
However, we lacked the strong exchange property for supermatroids, which is a key technique of the optimization problem.
Therefore, it was unknown whether the greedy algorithm has an approximation guarantee for a downward DR-submodular function maximization under a modular supermatroid constraint.
Here, using the strong exchange property, we give a partial answer:
The greedy algorithm has a $1/2$-approximation gurantee for a \emph{strong DR-submodular} maximization problem over atomic modular lattices under a modular supermatroid constraint.

The strong DR-submodularity is another generalization of DR-submodularity and stronger than downward DR-submodularity~\cite{nakashima2018subspace}.
\begin{definition}[Strong DR-submodularity]
Let $\lat$ be a modular lattice.
A function $f \colon \lat \to \mathbb{R}$ is a \emph{strong DR-submodular} if
\begin{align}
    f(Y \lor a) - f(Y) \le f(X \lor \underline a) - f(X),
\end{align}
for all $X, Y \in \lat$ with $X \le Y$, $a \in \adm{Y}$, and $\underline a \le a$ with $\underline a \in \adm{X}$.
\end{definition}

Two facts regarding a strong DR-submodularity should be noted~\cite{nakashima2018subspace}.
First, a strong DR-submodularity implies a bidirectional DR-submodularity, however, the converse is not true.
In addition, a strong DR-submodularity is equivalent to downward directional DR-submodularity on a distribuitve lattice.

\begin{figure}
\begin{algorithm}[H]
\caption{Greedy algorithm for modular supermatroid constrained problem.}
\label{alg:greedy_DR}
\begin{algorithmic}[1]
\State{$X = \bot$}
\For{$i = 1, \ldots, k$}
\State{Let $a_i \in \displaystyle \argmax_{a \in \adm{X}, X \lor a \in \mathcal{\ind}} f(X \lor a)$}
\State{$X \leftarrow X \lor a_i$}
\EndFor
\State{\Return $X$}
\end{algorithmic}
\end{algorithm}
\end{figure}

\begin{theorem}
\label{thm:greedy-DR-matroid}
Consider a maximization problem of a monotone strong DR-submodular function $f$ over atomic modular lattice $\lat$ under a modular supermatroid constraint $\ind \subseteq \lat$.
Let $k$ be the rank of $\ind$.
Algorithm \ref{alg:greedy_DR} then has an approximation ratio of $1/2$.
\end{theorem}
We describe the difficulties generalizing to general modular lattices and directional DR-submodular functions in Remark~\ref{rem:downward-max-generalization-hard}.
Such generalizations are left as for future studies.

When the ``linearity'' of the objective function is characterized by the curvature as in a Boolean lattice case~\cite{conforti1984submodular}, we can guarantee a refined approximation.
\begin{definition}[\cite{nakashima2018subspace}]
A monotone bidirectional DR-submodular function $f$ has a curvature $c$ if, for all $X\in \mathcal{L}$, $a \in \adm{X}$, and minimal $\underline{a} \le a$, 
\begin{align}
    f(X \lor a) - f(X) \ge (1-c) f(\underline{a}).
\end{align}
\end{definition}

\begin{theorem}
\label{thm:greedy-DR-matroid-curvature}
Consider a maximization problem of a monotone bidirectional DR-submodular function over a  (possibly non-atomic) modular lattice $\lat$ with curvature $c$ under modular supermatroid constraint $\ind \subseteq \lat$.
Algorithm~\ref{alg:greedy_DR} has an approximation ratio of $1-c$.
\end{theorem}

\subsection{Dual Matroids}
\label{subsec:modulardual}
In this section, we define a dual matroid of a modular supermatroid.
This duality is used to prove Proposition \ref{prop:rank-upward}, and will have an independent interest.

Let $\bar{(\cdot)} \colon \lat \to \lat^*$ be an order-reversing isomorphism.
Let $\base$ be the set of bases on $\lat$.
Then, the corresponding dual matroid is a modular supermatroid on $\lat^*$ whose set of bases $\base^*$ is
\begin{align}
    \base^* = \{\bar{B} \mid B \in \base\}.
\end{align}
The dual matroid is indeed a modular supermatroid because the base Axiom is self-dual (i.e., invariant for an order-reversal) and $\bar{(\cdot)}$ is an order-reversing isomorphism.

\begin{remark}
We can always define $\bar{(\cdot)}$ by $\bar{X} = X$ for all $X$.
For some special classes of lattices, we can take a natural $\bar{(\cdot)}$.
For a set lattice over finite set $E$ (i.e., finite boolean lattice), 
$\lat \cong \lat^*$, and 
we can define $\bar{X} = E \setminus X$.
Similarly, for a vector lattice $\lat(\R^d)$, 
we have $\lat(\R^n) \cong \lat(\R^n)^*$ and
can define a map $\bar{X} = X^\perp$.
For a distributive lattice, we can define $\bar{(\cdot)}$ as  the complement operation on a Birkhoff's  representation. 
\end{remark}

\begin{example}
Let $\lat$ be a set lattice over a finite set $E$ and $\bar{X} = E \setminus X$.
Then, $\base^*$ defines the usual dual matroid.
\end{example}

\begin{example}
Let $\bar{X} = X$.
In this setting, an element $X \in \lat^*$ is an independent set of the dual matroid if $X \ge B$ for some base $B$ of the primal modular supermatroid.
\end{example}

We characterize the rank function of the dual matroid as follows:
\begin{proposition}
\label{prop:dualrank}
Let $r$ be the rank function of a modular supermatroid.
Let $r^*$ be the rank function of the dual matroid.
Then,
\begin{align}
    r^*(\bar{X}) =  r(X) + (|\top| - |X|) - r(\top),
\end{align}
where $|\cdot|$ is the height of $\lat$.
\end{proposition}

\subsection{Examples}

Here, we give examples of modular supermatroids.
We use our rank axiom to show the structure is in fact modular supermatroid.

\begin{example}[Uniform Matorid]
Let $\mathcal{L}$ be a modular lattice.
For an integer $k$, the set $\mathcal{I} = \{ I \in \mathcal{L}: |I| \le k \}$ of the lattice elements whose height is at most $k$ forms a modular supermatroid 
because  the rank function of $\mathcal{I}$ is $h(I) = \min \{ |I|, k \}$ and satisfies the rank axiom.
\end{example}
\begin{example}[Partition Matorid]
Let $\mathcal{L}_1, \dots, \mathcal{L}_n$ be modular lattices.
Then their direct product $\mathcal{L} = \mathcal{L}_1 \times \cdots \times \mathcal{L}_n := \{ (X_1, \dots, X_n) \mid X_1 \in \mathcal{L}_1, \dots, X_n \in \mathcal{L}_n \}$ forms a modular lattice with respect to the component-wise partial order.
Let $k_1, \dots, k_n$ be integers.
Then, the set $\mathcal{I} = \{ (I_1, \dots, I_n) \mid |I_1| \le k_1, \dots, |I_n| \le k_n \}$ forms a modular supermatroid on $\mathcal{L}$ whose rank function is $h((I_1, \dots, I_n)) = \min \{ |I_1|, k_1 \} + \dots + \min \{ |I_n|, k_n \}$.
\end{example}

\begin{example}[Linear Matroid]
Let $\lat$ be a sublattice of $\lat(\R^d)$ such that for all $X \in \lat$, the height of $X$ in $\lat$ is the same as that of in $\lat(\R^d)$. 
A function $A \colon \R^d \to \R^k$ is said to be \emph{collinear-preserving} if $A(x)$, $A(y)$, and $A(z)$ are collinear (i.e., three points are on a certain straight line) for all collinear $x$, $y$, and $z$ in $\R^d$.
A \emph{linear supermatroid} $\ind_A \subseteq \lat$ represented by a collinear-preserving map $A$ is defined by
\begin{align}
    \ind_A = \{ X \in \lat \mid \dim \mathrm{span}(A(X)) = \dim X\},
\end{align}
where 
$\dim X$ is the dimension of vector space $X$, and $\mathrm{span}(Y)$ is the linear span of $Y \subseteq \R^d$.
Later, we see this definition generalizes the usual linear matroid.

The rank function $r_A$ of $\ind_A$ is given by
\begin{align}
    r_A(X) = \dim \mathrm{span}(A(X)).
\end{align}
We prove that $r_A$ satisfies (R1), (R2), and (R3).
The condition (R1) trivially holds.
The condition (R2) follows from the collinear-preserving property of $A$.
We finally prove (R3).
Let $X, Y \in \lat$ with $X \subseteq Y$ and $v \in \adm{Y}$.
If $r_A(Y \lor v) = r_A(Y)$, the downward DR-submodularity trivially holds.
Otherwise, there exists $v' \equiv v \bmod Y$ such that $\dim \mathrm{span}(A(Y \lor v')) = \dim \mathrm{span}(A(Y)) + 1$.
Because $X \subseteq Y$, we have $\dim \mathrm{span}(A(X \lor v')) = \dim \mathrm{span}(A(X)) + 1$, which implies the downward DR-submodularity. 

The linear supermatroid is a generalization of the linear matroid.
Let $V = \{1,2,\dots, d\}$.
For a set of linearly independent vectors $\{e_i \mid i \in V \}$, we define $\lat = \{\mathrm{span}(\{e_i \mid i \in S\}) \mid S \subseteq V\}$, which can be identified as $2^V$.
A collinear-preserving map $A$ yields a $k \times |V|$ matrix $A'$ as $A'_{ij}= A(e_i)_j$, where $x_j$ is the $j$-th component of $x \in \R^k$.
Then, the linear supermatroid $\ind_A$ is identified as a linear matroid represented by $A'$ over $2^V$.
\end{example}

\section{Supermatroids on Lower Locally Modular Lattices}
\label{sec:llm}

It is a natural question whether the relation between the axioms remains true on a class of lattices that is more general than the modular lattices.
Here, we consider lower-locally modular lattice, which is a common generalization of modular lattice and lower-locally distributive lattice.

We show that the independence axiom characterizes the matroids; however, all other axioms do not.
All proofs are given in Section~\ref{sec:proof}.

\subsection{Independence Axiom}

A supermatroid on lower locally modular lattices is also characterized by the independence axiom.
More generally, a supermatroid on arbitrary lattices is characterized by the independence axiom.
\begin{theorem}[Height Axiom $\Leftrightarrow$ Independence Axiom]
\label{thm:llm-height-ind}
Let $\mathcal{L}$ be a lattice.
\begin{itemize}
    \item[(1)] A supermatroid $\mathcal{I}$ on $\mathcal{L}$ satisfies (I1) and (I2).
    \item[(2)] Conversely, if $\mathcal{I} \subseteq \mathcal{L}$ satisfies (I1) and (I2), it is a supermatroid on $\mathcal{L}$.
\end{itemize}
\end{theorem}

In a modular lattice, the independence axiom can be replaced with a local version (I2l) as shown in Proposition~\ref{prop:localaug}.
However, (I2l) is too weak to characterize supermatroids on lower locally modular, and even worse, lower locally distributive lattices~\cite{fujishige2007matroids}.
We consider the lattice in Figure~\ref{fig:local-augmentation-lld}, where $\mathcal{I}$ is indicated by the black points in the figure.
This lattice is lower locally distributive, and  satisfies (I1) and (I2l).
However, it does not satisfy (H1); thus, it is not a supermatroid.

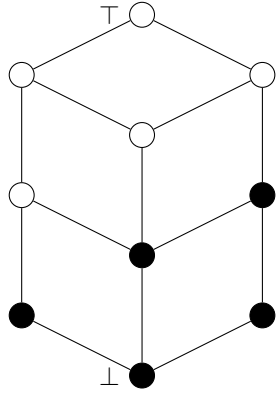
\begin{figure}[tb]
\centering
\begin{tikzpicture}[scale=0.8]
\node[draw, circle, fill, label=west:{$\bot$}] at (0,0) (bot) {};
\node[draw, circle, fill] at (0,2) (b1) {};
\node[draw, circle] at (0,4) (b2) {};
\node[draw, circle, fill] at (-2,1) (a1) {};
\node[draw, circle] at (-2,3) (a2) {};
\node[draw, circle] at (-2,5) (a3) {};
\node[draw, circle, fill] at (2,1) (c1) {};
\node[draw, circle, fill] at (2,3) (c2) {};
\node[draw, circle] at (2,5) (c3) {};
\node[draw, circle, label=west:{$\top$}] at (0,6) (top) {};
\draw[-] (bot)--(a1);
\draw[-] (bot)--(b1);
\draw[-] (bot)--(c1);
\draw[-] (a1)--(a2);
\draw[-] (a2)--(a3);
\draw[-] (c1)--(c2);
\draw[-] (c2)--(c3);
\draw[-] (b1)--(b2);
\draw[-] (b1)--(a2);
\draw[-] (b1)--(c2);
\draw[-] (b2)--(a3);
\draw[-] (b2)--(c3);
\draw[-] (top)--(a3);
\draw[-] (top)--(c3);
\end{tikzpicture}
\caption{Proposition~\ref{prop:localaug} does not hold on lower locally distributive lattice.}
\label{fig:local-augmentation-lld}
\end{figure}

Fujishige, Koshevoy, and Sano~\cite{fujishige2007matroids} called the structure satisfying (I1) and (I2) (applied to the lattice of convex sets) a strict cg-matroid,
and then  defined the following weaker version of the independence axiom to define a matroid on the lattice of convex sets.
\begin{definition}[Weak Independence Axiom] \ 
\begin{description}
\item[(I1)] $\mathcal{I}$ is a non-empty ideal.
\item[(I2w)] For any $I_1, I_2 \in \ind$ such that $I_2$ is maximal in $\ind$ and $|I_1| < |I_2|$, there exists $J$ such that $I_1 < J \le I_1 \lor I_2$.
\end{description}
\end{definition}
Clearly (I2) implies (I2w), but the converse does not hold in general.
Here, we show that the converse holds when the lattice is modular.
\begin{theorem}
\label{thm:weak-USM}
Let $\mathcal{I} \subseteq \lat$ satisfying (I1).
Then, (I2) and (I2w) is equivalent if $\lat$ is modular.
\end{theorem}

\begin{remark}
(I2w) and (I2l) are not equivalent because the example in Figure~\ref{fig:local-augmentation-lld} satisfies (I2l) but not (I2w).
\end{remark}

\subsection{Middle Axiom (Base Axiom)}

If the underlying lattice is a lower-locally distributive lattice, the middle axiom does not characterize matroids.
More precisely, we prove that the middle axiom is equivalent to the \emph{weak} independence axiom in general lattices.

\begin{theorem}[Middle Axiom $\Leftrightarrow$ Weak Independence Axiom]
\label{thm:llm-base-wind}
Let $\mathcal{L}$ be a lattice and $\mathcal{I} \subseteq \lat$ satisfying (I1) and (I2w).
Then, the set of maximal elements $\mathcal{B}$ of $\mathcal{I}$ satisfies (B1) and (B2).
Conversely, if $\mathcal{B}$ satisfies (B1) and (B2), then $\mathcal{I} = \{ I \in \mathcal{L} \mid \exists B \in \mathcal{B}, I \le B \}$ satisfies (I1) and (I2w).
\end{theorem}

\subsection{Rank Axiom}

The rank axiom does not characterize matroids if an underlying lattice is lower-locally modular (or lower-locally distributice); 
the counterexample is given in Example~\ref{ex:lld-counterexample}.

\begin{example}
\label{ex:lld-counterexample}
\begin{figure}
    \centering
    \begin{tikzpicture}
    \node at (0,0) (0) {$\bot$};
    \node at (0,2) (d) {d};
    \node at (0,4) (e) {e};
    \node at (0,6) (1) {$\top$};
    \node at (-2,1) (a) {a};
    \node at (-2,3) (b) {b};
    \node at (-2,5) (c) {c};
    \node at (2,1) (f) {f};
    \node at (2,3) (g) {g};
    \node at (2,5) (h) {h};
    \draw (0)--(d);
    \draw (d)--(e);
    \draw (0)--(a);
    \draw (0)--(f);
    \draw (d)--(b);
    \draw (d)--(g);
    \draw (e)--(c);
    \draw (e)--(h);
    \draw (a)--(b);
    \draw (b)--(c);
    \draw (f)--(g);
    \draw (g)--(h);
    \draw (c)--(1);
    \draw (h)--(1);
    \end{tikzpicture}
    \caption{Counterexample: Local augumentation  (I2l) does not imply (I2) on lower locally modular lattices.}
    \label{fig:lld-counterexample}.
\end{figure}
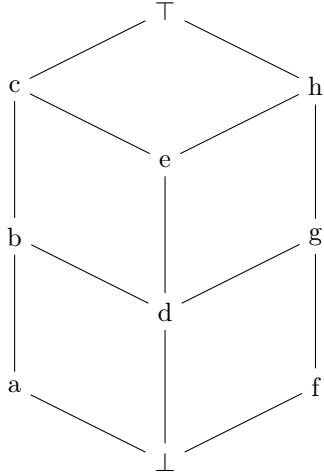

Consider the lattice shown in Figure~\ref{fig:lld-counterexample}.
We can easily check that this lattice is lower locally distributive.
Consider a rank function $r$ defined by $r(\perp) = 0$,
$r(a) = r(d) = r(f) = 1$,
$r(b) = r(e) = r(g) = r(h) = 2$, and
$r(c) = r(\top) = 3$.
Then, $r$ satisfies the rank axiom and the rank of the supermatroid is 3.
However, $g$ is a height-2 maximal element of the supermatroids, which contradicts (H2).
\end{example}

We prove that the independence axiom implies the rank axiom if the lattice is lower-semimodular, and the converse holds if the lattice is upper-semimodular.

\begin{theorem}[Independece Axiom $\Rightarrow$ Rank Axiom  on LSM; Rank Axiom $\Rightarrow$ Independence Axiom on USM] \ 
\label{thm:ind-rank}
\begin{itemize}
    \item[(1)] Let $\ind$ be a supermatroid on a lower semimodular lattice. Then, $r(X) := \min \{|I| \mid I \le X, I \in \ind\}$ satisfies (R1), (R2), and (R3).
    \item[(2)] Conversely, let $r$ be a function over an upper semimodular lattice satisfying (R1), (R2), and (R3). Then, $\ind := \{X \in \lat \mid r(X) = |X|\}$ is a supermatroid.
\end{itemize}
\end{theorem}

We tried to characterize the rank function of supermatroids on lower-locally modular lattices by modifying (R3). 
We tried the followings attempts.
\begin{itemize}
    \item Replace the downward-DR submodularity to the upward-DR submodularity or the bidirectional DR-submodularity.
    \item Replace the downward-DR submodularity to a modified version of the downward-DR submodularity.
\end{itemize}
Unfortunately, both attempts failed;
hence, it is an open problem to characterize the rank function of lower-locally modular lattices in terms of a submodular-like property.

Below, we explain the failure of our attempts.
The first modification ensures that the rank axiom implies the independence axiom
because of Theorem~\ref{thm:ind-rank};
however, the converse does not hold; see the following example.

\begin{example}
Consider a lower locally distributive lattice $\lat$ shown in Figure~\ref{fig:counterexample-upward}.
The elements in $\ind$ are shown by the black dots.
Although $\ind$ is a supermatroid, the rank function $r$ is not upward DR-submodular.
By taking $Y \lor b := X$, $Y := W$, and $X := Y \lor q$ in~(\ref{eq:downwardDR}), we can easily see that $r$ is not downward DR-submodular on $\lat^*$. 
\begin{figure}
    \centering
    \begin{tikzpicture}
    \node[draw,circle,fill] at (0,0) (bot) {};
    \node[draw,circle,label=west:{$X$}] at (-2,1) (1) {};
    \node[draw,circle,label=north:{$p$}] at (0,2) (2) {};
    \node[draw,circle,fill] at (2,1) (3) {};
    \node[draw,circle, label=west:{$W$}] at (-2,3) (4) {};
    \node[draw,circle,label={$Y$}] at (0,4) (5) {};
    \node[draw,circle] at (2,3) (6) {};
    \node[draw,circle,fill,label=east:{$q$}] at (4,2) (7) {};
    \node[draw,circle] at (4,4) (8) {};
    \node[draw,circle,label={$Y \lor q$}] at (2,5) (9) {};
    
    \draw[-] (bot)--(1);
    \draw[-] (bot)--(2);
    \draw[-] (bot)--(3);
    \draw[-] (1)--(4);
    \draw[-] (3)--(6);
    \draw[-] (2)--(4);
    \draw[-] (2)--(6);
    \draw[-] (4)--(5);
    \draw[-] (6)--(5);
    \draw[-] (3)--(7);
    \draw[-] (6)--(8);
    \draw[-] (7)--(8);
    \draw[-] (9)--(8);
    \draw[-] (9)--(5);
    \end{tikzpicture}
    \caption{Counterexample: A rank function is not upward DR-submodular. The elements in $\ind$ is shown by black dots.}
    \label{fig:counterexample-upward}
\end{figure}
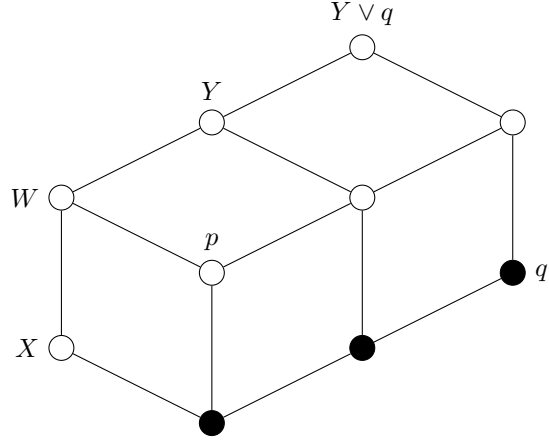
\end{example}

For the second modification, we introduce the following strengthening version of the DR-submodularity. 
We say that $p \in \mathcal{L}$ is a \emph{co-extreme point} if $p$ is join-irreducible and $X \lor p \succ X$.
We denote by $\coex{X}$ the set of co-extreme points.

\begin{definition}[Downward DR-Submodular']
A function is downward DR-submodular' if for all $X, Y \in \mathcal{L}$ and $q \in \coex{Y}$, 
\begin{align}
\label{eq:downwardDR}
    f(Y \lor q) - f(Y) \le \max_{\substack{Z \equiv q \bmod Y \\ \coex{X} \cap \{ p : p \le Z \} \neq \emptyset}} \min_{p \in \coex{X}, p \le Z} f(X \lor p) - f(X).
\end{align}
\end{definition}
This definition is distinguished from the original by the prime, and is equivalent to the original definition on modular lattices.

\begin{lemma}
\label{lem:coexDR-equiv-DR-modular}
Let $\mathcal{L}$ be a modular lattice.
The downward DR-submodularity' is then equivalent to the downward DR-submodularity.

\end{lemma}

We replace the (R3) in the rank axiom by the following (R3s):
\begin{description}
\item[(R3s)] $\rho$ is a downward DR-submodular' function.
\end{description}
Then, the modified properties implies the independence axiom.
\begin{theorem}
\label{thm:coexDR-imply-Ind}
Let $\mathcal{L}$ be a lower locally modular lattice, and $\mathcal{I}$ be a supermatroid on $\mathcal{L}$.
Then, if a function $\rho \colon \mathcal{L} \to \mathbb{R}$ satisfies (R1), (R2), and (R3s), then $\mathcal{I} := \{ X \in \mathcal{L} : \rho(X) = |X| \}$ forms a supermatroid.
\end{theorem}
However, the converse does not hold as shown in the following example.
\begin{example}
Consider the same example in the previous example (Figure~\ref{fig:lld-counterexample}).
Recall that this lattice is lower locally distributive, and $\ind$ shown by black dots is a supermatroid.
The rank function is not downward DR-submodular'.
Take $X, Y, p, q$ as in Figure~\ref{fig:lld-counterexample}.
Then, $\coex{X} = \{p\}$ and $f(X \lor p) - f(X) = 0$,
i.e., the independence axiom does not implies the (replaced version of the) rank axiom.
\end{example}

These results indicate that it is difficult to characterize the rank function of a supermatroid over a lower locally modular lattice by submodularity.

Note that if we do not stick to the submodularity of the rank function, we can define the following  ``rank axiom'' that characterize matroids.
Note that this definition is motivated by Sano~\cite[Theorem~1.2]{sano2008rank}, which is a characterization of matroids on lower-locally distributive lattices.

\begin{definition}[Rank Axiom (without Submodularity)] \ 
\begin{description}
\item[(R1)] $\rho(\bot) = 0$
\item[(R2)] $\rho(X \lor p) - \rho(X) \in \{0, 1\}$ for all $p \in \coex{X}$.
\item[(R3')] For any $X, Y \in \mathcal{L}$ such that $X \subseteq Y$ and $\rho(X) =|X| < \rho(Y)$,
    there exists $e \in \coex{X} \land Y$ such that $\rho(X \cup \{e\}) = \rho(X) + 1$.
\end{description}
\end{definition}
\begin{theorem}
\label{thm:sanotype-rank-to-ind}
Let $\ind$ be a supermatroid on lower locally modular lattice $\lat$.
Then, $r(X) := \max\{|I| \mid I \in \ind, I \le X\}$ satisfies (R1), (R2), and (R3').
Conversely, let $r \colon \lat \to \mathbb{Z}$ be a function satisfying (R1), (R2), and (R3').
Then, $\ind := \{ X \in \lat \mid r(X) = |X|\}$ satisfies (I1) and (I2).
\end{theorem}

\subsection{Dependendence Axiom}

The dependence axiom does not characterize matroids if an underlying lattice is lower-locally modular.
We prove that the dependence axiom is equivalent to the ``local'' independence axiom.

\begin{theorem}[Dependence Axiom $\Leftrightarrow$ Local Independence Axiom]
\label{thm:llm-dep-ind}
\begin{itemize}
\item[(1)] Let $\ind$ be a supermatroid. Then, $\dep = \lat \setminus \ind$ satisfies (D1), (D2), and (D3).
\item[(2)] Conversely, let $\dep \subseteq \lat$ satisfying (D1), (D2), and (D3). Then, $\ind := \lat \setminus \dep$ satisfies (I1) and (I2l).
\end{itemize}
\end{theorem}

\section{Proofs}
\label{sec:proof}

In this section, we give the omitted proofs.

\subsection{Properties of  Lattices}
Here, we summarize the properties of the lattices that are necessary for better understanding the following proofs.
We state that $p \in \mathcal{L}$ is an \emph{extreme point} if $p$ is a join-irreducible and there exists $\mathring{X} \prec X$ such that $X = \mathring{X} \lor p$. 
We denote by $\ex{X} = \{ (x, \mathring{B}) \mid \mathring{B} \prec \mathring{B} \lor x = B \} $ the set of extreme points with its remaining part.
We can see that $p \in \coex{X}$, and thus $(p, X) \in \ex{X \lor p}$.

\begin{lemma}
Let $\lat$ be an upper semimodular lattice.
Then, for all $X \in \lat$ and $a \in \adm{X}$, $X  \prec X \lor a$.
\end{lemma}
\begin{proof}
$a \in \adm{X}$ indicates $X \land a \prec a$. By the upper semimodularity, this implies that $X  \prec X \lor a$.
\end{proof}

\begin{lemma}
\label{lem:precheight}
Let $\lat$ be a lower or upper semimodular lattice.
Then, $X \prec Y$ if and only if $X \le Y$ and $|X| + 1 = |Y|$.
\end{lemma}
\begin{proof}
Because the height is well-defined in a lower or upper semimodular lattice, this lemma trivially holds.
\end{proof}

\begin{lemma}
\label{lem:subadm}
Let $\lat$ be a lattice.
For any $X, Y \in \lat$ and $a \in \adm{Y}$ with $a \not \le X$, there exists $b \in \adm{X}$ such that $b \le a$.
\end{lemma}
\begin{proof}
We construct $b$ as follows.
Initially, we set $b_0 = a$.
For each $i$, if there exists $Z \prec b_i$ with $Z \not \le X$, we set $b_{i+1} = Z$ and continue the process; otherwise, we terminate the process and obtain $b = b_i$.
Note that $b \neq \bot$ because $\bot \le X$.
Then, $\{ Z \in \lat : Z \prec b \}$ are maximal elements in $X \land b$, and because $X \land b$ must be determined uniquely, there exists a unique $Z \prec b$ such that $Z = X \land b$.
This means $b \in \adm{X}$.
\end{proof}

\begin{lemma}
\label{lem:prec}
Let $\lat$ be a lower semimodular lattice.
For any $X, Y \in \lat$ with $X \prec Y$ and $Z \le Y$ with $Z \not \le X$, we have $X \land Z \prec Z$.
\end{lemma}
\begin{proof}
Because $Z \not \le X$ and $X, Z \le Y$, we have $X < X \lor Z \le Y$. 
Because $X \prec Y$, it must be $X \lor Z = Y$; thus, $X \prec X \lor Z$.
By the lower submodularity, this implies $X \land Z \prec Z$.
\end{proof}

\begin{lemma}[Chernoff Property]
Let $\mathcal{L}$ be a lower semimodular lattice.
Let $X, Y \in \mathcal{L}$ such that $X \le Y$, and $(y, \mathring{Y}) \in \ex{Y}$. 
If $y \in X$, then $(y, \mathring{Y} \land X) \in \ex{X}$.
\end{lemma}
\begin{proof}
Based on this assumption, we have $X \not \le \mathring{Y}$; otherwise, $y \le X \le \mathring{Y}$, which contradicts $\mathring{Y} \lor y \succ \mathring{Y}$.
Thus, $X \lor \mathring{Y} \succ \mathring{Y}$. 
By the lower semimodularity, we have $X \land \mathring{Y} \prec X$.
By $y \not \le X \land \mathring{Y}$ and $y \le X$, we have $X \land \mathring{Y} = X$.
\end{proof}
\begin{lemma}
For any $(z, Z) \in \ex{X \lor Y}$, the element $z$ is extreme point for $X$ and $Y$.
\end{lemma}
\begin{proof}
By definition, 
$Z \prec X \lor Y$. 
Based on a lower locally modularity, we have
$Z \land X \prec X$ and
$Z \land Y \prec Y$.
Because $z \not \le X \land Z$ and $z \not \le Y \land Z$, we proved the statement.
\end{proof}

\begin{lemma}
\label{lem:existence-of-join-irreducibles}
Let $X < Y$.
Then, there exists $q \le Y$ such that $q \in \coex{X}$.
\end{lemma}
\begin{proof}
Let $X'$ be $X \prec X' \le Y$.
Then, there exists a join irreducible $q$ that is less than $X'$ and not less than $X$ because we otherwise have $X' = \bigvee \{ q \in J : q \le X' \} \le X$, which contradicts $X \prec X'$.
Because $X \lor q = X' \succ X$, $q \in \coex{X}$.
\end{proof}

The following lemma is useful in a lower locally modular lattice.
\begin{lemma}[Ladder Lemma]
\label{lem:ladder}
Let $\mathcal{L}$ be a lower locally modular lattice.
Let $X_0 \prec X_1 \prec \dots \prec X_m$ be a maximal chain between $X_0$ and $X_m$, and $Z \in \mathcal{L}$ be an element such that $Z \prec X_m$.
There then exists $l \in \mathbb{Z}$ such that 
$X_i \land Z \prec X_i$ for $i = l+1, \dots, m$ and $X_{l-1} = X_{l} \land Z$ (see Figure~\ref{fig:ladder}).
\begin{figure}[tb]
\begin{minipage}{.48\textwidth}
\centering
\begin{tikzpicture}[scale=0.8]
\node[draw, circle, label=west:{$X_5$}] at (5,5) (X5) {};
\node[draw, circle, label=west:{$X_4$}] at (4,4) (X4) {};
\node[draw, circle, label=west:{$X_3$}] at (3,3) (X3) {};
\node[draw, circle, label=west:{$X_2$}] at (2,2) (X2) {};
\node[draw, circle, label=west:{$X_1$}] at (1,1) (X1) {};
\node[draw, circle, label=west:{$X_0$}] at (0,0) (X0) {};
\node[draw, circle, label=east:{$Z = X_5 \land Z$}] at (6,4) (Z5) {};
\node[draw, circle, label=east:{$X_4 \land Z$}] at (5,3) (Z4) {};
\node[draw, circle, label=east:{$X_3 \land Z$}] at (4,2) (Z3) {};
\node[draw, circle, label=east:{$X_2 \land Z$}] at (3,1) (Z2) {};
\node[draw, circle, label=east:{$X_1 \land Z$}] at (2,0) (Z1) {};
\node[draw, circle, label=east:{$X_0 \land Z$}] at (1,-1) (Z0) {};
\draw[-] (X5)--(X4);
\draw[-] (X4)--(X3);
\draw[-] (X3)--(X2);
\draw[-] (X2)--(X1);
\draw[-] (X1)--(X0);
\draw[-] (Z5)--(Z4);
\draw[-] (Z4)--(Z3);
\draw[-] (Z3)--(Z2);
\draw[-] (Z2)--(Z1);
\draw[-] (Z1)--(Z0);

\draw[-] (X5)--(Z5);
\draw[-] (X4)--(Z4);
\draw[-] (X3)--(Z3);
\draw[-] (X2)--(Z2);
\draw[-] (X1)--(Z1);
\draw[-] (X0)--(Z0);
\end{tikzpicture}
\end{minipage}
\qquad
\begin{minipage}{.48\textwidth}
\centering
\begin{tikzpicture}[scale=0.8]
\node[draw, circle, label=west:{$X_5$}] at (5,5) (X5) {};
\node[draw, circle, label=west:{$X_4$}] at (4,4) (X4) {};
\node[draw, circle, label=west:{$X_3$}] at (3,3) (X3) {};
\node[draw, circle, label=east:{$Z = X_5 \land Z$}] at (6,4) (Z5) {};
\node[draw, circle, label=east:{$X_4 \land Z$}] at (5,3) (Z4) {};
\node[draw, circle, label=east:{$X_2 = X_2 \land Z = X_3 \land Z$}] at (4,2) (X2) {};
\node[draw, circle, label=east:{$X_1 = X_1 \land Z$}] at (3,1) (X1) {};
\node[draw, circle, label=east:{$X_0 = X_0 \land Z$}] at (2,0) (X0) {};
\draw[-] (X5)--(X4);
\draw[-] (X4)--(X3);
\draw[-] (X3)--(X2);
\draw[-] (X2)--(X1);
\draw[-] (X1)--(X0);
\draw[-] (Z5)--(Z4);
\draw[-] (X5)--(Z5);
\draw[-] (X4)--(Z4);
\draw[-] (X2)--(Z4);
\end{tikzpicture}
\end{minipage}
\caption{Ladder Lemma. Each line indicates the cover relation ($\prec$).}
\label{fig:ladder}
\end{figure}
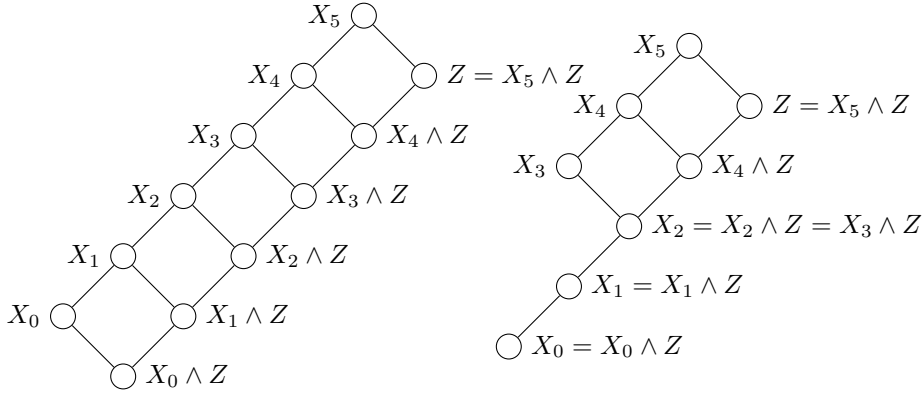
\end{lemma}
\begin{proof}
We prove the claim through an induction based on the length of the chain.
Suppose $m = 1$. 
Then, $X_1$ has $X_0$, and $Z = X_1 \land Z$ as the covering elements. 
If $X_0 = X_1 \land Z$, then the claim holds.
Otherwise, by the lower local modularity at $X_1$, we see that $X_0 \land (X_1 \land Z) \prec X_1 \land Z$ is covered by both $X_0$ and $X_1 \land Z$.
Thus, the claim holds.

Now, we suppose $m \ge $.
Based on the induction hypothesis applied to the subchain $X_1 \prec \dots \prec X_m$, we obtain the ladder from $X_1$.
If $X_1 = X_1 \land Z$, then $X_0 = X_0 \land X_1 = X_0 \land (X_1 \land Z) = X_0 \land Z$; therefore, the claim holds.
Otherwise, by the induction hypothesis to $X_0 \prec X_1$ and $X_1 \land Z \prec X_1$, we prove the claim.
\end{proof}

\subsection{Proofs for Independence Axiom}
\begin{proof}[Proof of Theorem~\ref{thm:modular-ind-height} and Theorem~\ref{thm:llm-height-ind}]
It suffices to prove Theorem~\ref{thm:llm-height-ind}.

(H1) and (I1) are the same.
We prove the equivalence of (H2) and (I2) under (H1).

\noindent ($\Rightarrow$) \ 
Let $I_1, I_2 \in \mathcal{I}$ with $|I_1| < |I_2|$. 
We choose any maximal element $J \in \ind^{I_1 \lor I_2}$ such that $J \ge I_1$.
By (H1), $|J| \ge |I_2| > |I_1|$. Thus, $J$ satisfies the requirement.

\noindent ($\Leftarrow$) \ 
If there exists $X \in \mathcal{L}$ and maximal $I_1, I_2 \in \mathcal{I}^X$ with $|I_1| < |I_2|$, by (I1), there exists $J \in \mathcal{I}$ such that $I_1 < J \le I_1 \lor I_2$.
However, because $J \le I_1 \lor I_2 \le X$, $J \in \mathcal{I}^X$, which contradicts the maximality of $I_1$.
\end{proof}

\begin{proof}[Proof of Proposition~\ref{prop:localaug}]
Let $I_1, I_2 \in \ind$ such that $|I_1| \le |I_2|$.
We prove that there exists $J \in \ind$ such that $I_1 < J \le I_1 \lor I_2$ by the induction on $k := |I_1 \lor I_2| - |I_2|$.
If $k = 0$, then we can choose $J = I_2$.
If $k \ge 1$, then we pick $I_2 < K \prec I_1 \lor I_2$. 
Here, $I_1 \not \le K$ (otherwise, $I_1 \lor I_2 = K$, which contradicts $K \prec I_1 \lor I_2$);
hence, $I_1 \lor K \succ K$.
By the lower semimodularity, $I_1 \land K \prec I_1$.
We apply the induction hypothesis twice to obtain $K_1, K_2 \in \mathcal{I}$ such that $I_1 \land K \prec K_1 \prec K_2 \le K$.
Now, we see $I_1 \land K_2 \prec I_1$; hence, by the upper semimodularity, $I_1 \lor K_2 \succ K_2$. 
Therefore, by (I2l), we obtain the claim.
\end{proof}

To prove Theorem~\ref{thm:weak-USM}, we use the following lemma.
\begin{lemma}
\label{lem:prec-USM}
Let $\mathcal{L}$ be an upper semimodular lattice.
For any $X, Y, Z \in \mathcal{L}$ such that $X \prec Y$, we have $X \lor Z \preceq Y \lor Z$.
\end{lemma}
\begin{proof}
Let $W = Z \lor X$.
Because $X \le W \land Y \le Y$, the meet $W \land Y$ must be either $X$ or $Y$.
If it is $Y$, we have $Y \le W$.
Therefore, $Y \lor Z = (Y \lor X) \lor Z = Y \lor W = W = X \lor Z$.
Otherwise, i.e., if it is $X$, we have $W \land Y \prec Y$.
By the upper semimodularity, $W \prec W \lor Y = Z \lor Y$.
\end{proof}

\begin{proof}[Proof of Theorem~\ref{thm:weak-USM}]
Let $I_1$ and $I_2$ be $|I_1| < |I_2|$.
We prove that there exists $J$ such that $I_1 \prec J \le I_1 \lor I_2$ by the induction on $|I_2|$.
If $I_2$ is maximal, there is nothing to prove.
If $I_2$ is not maximal, we choose $I_2 \prec I'$ with $I' \in \mathcal{I}$ and apply the induction hypothesis to obtain $K \in \mathcal{I}$ such that $I_1 \prec K \le I_1 \lor I'$.
Because $|K| < |I'|$,
we can apply the induction hypothesis to $K$ and $J'$ and obtain $K' \in \mathcal{I}$ such that $K \prec K' \le K \lor I' \le I_1 \lor I'$.

By Lemma~\ref{lem:prec-USM}, $I_1 \lor I_2 \preceq I_1 \lor I'$.
If $I_1 \lor I_2 = I_1 \lor I'$ then $K \le I_1 \lor I_2$ and $K$ satisfies the requirement.
Otherwise,  $I_1 \lor I_2 \prec I_1 \lor I'$.
If $K' \le I_1 \lor I_2$, the element $K$ satisfies the requirement.
Otherwise, we have $I_1 \lor I_2 \prec (I_1 \lor I_2) \lor K'$.
By the lower semimodularity, $(I_1 \lor I_2) \land K' \prec K'$.
Thus, $(I_1 \lor I_2) \land K'$ satisfies the requirement.
\end{proof}

\subsection{Proofs for Base Axiom}

\begin{proof}[Proof of Lemma~\ref{lem:base-sameheight}]
Let $B_1, B_2 \in \mathcal{B}$ with $|B_1| \le |B_2|$.
We prove the lemma by the induction on $k = |B_1 \lor B_2| - |B_1|$.
If $k = 0$ then $B_1 = B_2$, which demostrates the claim.
For $k \ge 1$, we fix a chain $B_1 \prec \dots \prec Z \prec B_1 \lor B_2$.
Because $Z \lor B_2 \succ Z$, by the lower semimodularity, $B_2 \land Z \prec B_2$. 
By applying (B2) to $B_2 \land Z \le B_2$ and $B_1 \le Z$, we obtain $B \in \mathcal{B}$ such that $B_2 \land Z \le B \le Z$.
Here, $B_1 \lor B \le Z$; thus, $|B_1 \lor B| \le |Z| = |B_1 \lor B_2| - 1$. 
Therefore, by the induction hypothesis, $|B| = |B_1|$.
Because $B_2 \land Z < B$ (otherwise, $B$ is comparable with $B_2$, which contradicts the base incomparability), we have $|B| > |B_2 \land Z| = |B_2| - 1$; therefore, $|B| \ge |B_2|$. 
Hence, we obtain $|B_1| \ge |B_2|$.
\end{proof}

\begin{proof}[Proof of Theorem~\ref{thm:mid} and Theorem~\ref{thm:llm-base-wind}]
By Proposition~\ref{prop:localaug}, 
it suffices to prove Theorem~\ref{thm:llm-base-wind}.

\noindent ($\Leftarrow$). \ 
By definition, (B1) is trivial.
We prove (B2). 
Let $B_1, B_2, X, Y$ be as in (B2).
Because $\mathcal{I}$ is an ideal and $X \le B_1$, we have $X \in \mathcal{I}$.
We construct $B$ as follows.
Initially, we put $B = X$. 
If $B$ is a maximal in $\mathcal{I}$ then $B$ satisfies the requirement.
Otherwise, by (H2), $|B| < |B_1| = |B_2|$.
By (I2w), there exists $I \in \mathcal{I}$ such that $B < I \le B_1 \lor B_1$.
We replace $B$ by $I$ and continue the process. 
Because each step increases the height of $B$ by at least 1, it terminates in a finite step.

\noindent ($\Rightarrow$). \ 
By definition, (I1) is trivial. 
We prove (I2w).
Let $I_1, I_2 \in \mathcal{I}$ with $|I_1| < |I_2|$, where $I_2$ is maximal.
Because $I_2$ is maximal, $I_2 \in \base$.
By the definition of $\mathcal{I}$, there exists $B_1 \in \mathcal{B}$ such that $I_1 \le B_1$.
By (B2), there exists $B \in \mathcal{B}$ such that $I_1 \le B \le I_2$. 
By Lemma~\ref{lem:base-sameheight}, we have $|I_1| < |I_2| = |B|$; hence, $I_1 \neq B$.
Therefore, (I2w) holds.
\end{proof}

\subsection{Proofs for Rank Axiom}

\begin{proof}[Proof of  Theorem~\ref{thm:ind-rank-mod} and Theorem~\ref{thm:ind-rank}]
It suffices to prove Theorem~\ref{thm:ind-rank} because a modular lattice is both lower and upper semimodular.

\noindent (1). \ 
(R1) is trivial.
We prove (R2). 
It is easy to see that $r$ is an integer-valued non-decreasing function. 
Hence, it suffices to show that the increment $r(X) - r(X')$ is at most 1 for all $X \succ X'$.
Suppose that $r(X) = |I|$ for $I \in \ind$ and $I \le X$.
If $I \le X'$, then $r(X) - r(X') = 0$.
Otherwise, by Lemma~\ref{lem:prec}, $X' \land I \prec I$.
Because $X' \land I \in \ind$, we have $r(X) \ge |X' \land I| = |I| - 1 = r(X) - 1$, where the first equality follows from Lemma~\ref{lem:precheight}.

We finally prove (R3).
Let $X, Y \in \lat$ with $X \le y$ and $b \in \adm{Y}$.
If $r(Y \lor b) - r(Y) = 0$, the  downward DR-submodularity trivially holds.
Thus, by (R2), we only have to consider the case in which $r(Y \lor a) - r(Y) = 1$.
We pick a maximal chain $X = X_0 \prec X_1 \prec \dots \prec X_k = Y \prec X_{k+1} = Y \lor a$.
We then construct a corresponding chain $I = I_0 \preceq I_1 \preceq \dots \preceq I_k \prec I_{k+1}$ as follows.
Let $I_0 \in \ind$ be the maximal independent set of $X_0$.
For each $i = 1, \dots, k+1$,  we construct $I_i \in \ind$ such that $r(X_i) = |I_i|$ by the following rule.
If $r(X_i) = |I_{i-1}|$, we apply $I_i = I_{i-1}$.
Otherwise, there exists $J_i \in \ind$ with $J_i \le X_i$ and $|I_{i-1}| < |J_i|$. 
By (R2), $|J_i| = |I_{i-1}| + 1$.
In addition, by (I2), there exists $I_i \in \ind$ such that $I_{i-1} < I_i \le I_{i-1} \lor J_i \le X_i$. 
By comparing the heights, $r(X_i) = |I_i|$.
Because $X_k \prec X_{k+1}$, we have $I_{k+1} \not \le X_k$ through the construction of $I_{k+1}$.
By Lemma~\ref{lem:prec}, $X_k \lor I_{k+1} = X_{k+1}$.
In addition, because $I_k \prec I_{k+1}$, there exists $b \in \adm{I_k}$ such that $I_k \lor b = I_{k+1}$.
Therefore, we have $X_k \lor b = X_{k+1}$.
This shows that $b \in  a \bmod Y$.
We can see that $b \not \le X$ because $b \in \adm{Y}$. 
Let $b' \in \adm{X}$ be any $b' \le b$.
If $b' \in \adm{I}$ then we put $b'' = b'$.
Otherwise, by Lemma~\ref{lem:subadm}, we select $b'' \in \adm{I}$ such that $b'' \le b'$.
Because $I \lor b'' \le I_k \lor b$, we have $I \lor b'' \in \ind$.
In addition, because $I \lor b'' \le X \lor b'$, we have $r(X \lor b') - r(X) = |I \lor b''| - |I| \ge 1$.
This implies the downward DR-submodularity.

\noindent (2). \ 
(I1)
$\ind$ is non-empty because $\perp \in \ind$. By (R2), we can easily show that $\ind$ is an ideal.
(I2)
Let $X,Y \in \ind$ with $|X| < |Y|$.
Take a chain $X = X_0 \prec X_1 \prec \dots \prec X_k = X \lor Y$.
Because $r(X \lor Y) \ge r(Y) = |Y| > |X| = r(X)$, there is an $i$ such that $r(X_{i+1}) - r(X_i) = 1$.
Based on the downward DR-submodularity (R3), there exists $w \in \adm{X}$ such that $w \le X \lor Y$ and $r(X \lor w) - r(X) = 1$.
Because $r(X \lor w) = |X \lor w| = |X| + 1$, we have $X \lor w \in \ind$.
This implies (I2).
\end{proof}

\begin{proof}[Proof of Lemma~\ref{lem:coexDR-equiv-DR-modular}]
Note that, in a modular lattice, $\coex{X} = \adm{X}$. 
In addition, in a modular lattice, for any $Z$ with $Z \equiv q \bmod Y$, there exists $p \in \coex{X}$ with $p \le Z$. 
Thus, we can ignore the second constraint in the outer maximum in \eqref{eq:downwardDR}.

If $f$ satisfies the DR-submodularity \eqref{eq:downwarddr}, it clearly satisfies the DR-submodularity' \eqref{eq:downwardDR} because the range of the outer maximum in \eqref{eq:downwardDR} is larger than that of \eqref{eq:downwarddr}.

Conversely, if $f$ satisfies \eqref{eq:downwardDR}, we choose any $z \le Z$ such that $z \in \adm{Y}$.
Then, $z = q \bmod Y$. 
Thus, it is a candidate for the outer maximum in \eqref{eq:downwarddr}.
Because the range of the inner minimum in \eqref{eq:downwarddr} is smaller than that of \eqref{eq:downwardDR} for $z \le Z$, it satisfies \eqref{eq:downwarddr}.
\end{proof}

\begin{proof}[Proof of Proposition \ref{prop:rank-upward}]
In this proof, we use Proposition~\ref{prop:dualrank}, which is proved later.
We note that the proof of Proposition~\ref{prop:dualrank} does not depend on this proposition.
We first prove that (R3) implies (R3u).
We define an order-reversing map $(\cdot) \colon \lat \to \lat^*$ by $\bar{X} = X$.
Let $r^*$ be the rank function of the dual matroid.
By Proposition \ref{prop:dualrank}, 
\begin{align}
    r^*(X) =  r(X) + (|\top| - |X|) - r(\top),
\end{align}
where $|\cdot|$ is the height of $\lat$.
The rank axiom (R3) for the dual matroid implies that $r^*$ is a downward DR-submodular function on $\lat^*$.
By the definition of downward DR-submodularity, we can easily check that $r$ is also a downward DR-submodular function on $\lat^*$.
Hence, by definition, $r$ is upward DR-submodular.

We can prove that (R3u) implies (R3) by a similar argument.
\end{proof}

To prove Theorem~\ref{thm:coexDR-imply-Ind}, we need the following lemma.

\begin{lemma}
\label{lem:down}
Suppose that $\mathcal{L}$ is a lower locally modular lattice.
Let $X, Y \in \mathcal{L}$ with $X \le Y$ and $q \in \coex{Y}$. 
Then, there exists $Z \in \mathcal{L}$ and $p \in \coex{X}$ such that $Z = q \mod Y$ and $p \le Z$.
\end{lemma}
\begin{proof}
If $X = Y$, then we can choose $p = Z = q$.
Thus, we consider $X < Y$.
First, we choose an arbitrary chain from $q$ to $Y \lor q$ and select $Z \in \mathcal{L}$ as the second from the last in the chain, i.e., $q \le Z \prec Y \lor q$.
Now, we apply the Ladder lemma to the chain $X = X_0 \prec \dots \prec X_m = Y \lor Z$ and $Z$ to obtain the ladder.
If $X_0 \land Z \prec X_0$ (as shown on the left side of Figure~\ref{fig:ladder}), by Lemma~\ref{lem:existence-of-join-irreducibles}, there exists a join-irreducible $p$ such that $p \le X_1 \land Z$ and $p \not \le X$, which gives $X_0 \lor p = X_1 \succ X_0$.
Otherwise (the right of Figure~\ref{fig:ladder}), by Lemma~\ref{lem:existence-of-join-irreducibles}, there exists a join-irreducible $p$ such that $p \le X_1$ and $p \not \le X$, which gives $X_0 \lor p = X_1 \succ X_0$.
\end{proof}

\begin{proof}[Proof of Theorem~\ref{thm:coexDR-imply-Ind}]
We prove the independence axiom. 
(I1) holds by (R2).
We prove (I2).
Let $I_1, I_2 \in \mathcal{I}$ with $|I_1| < |I_2|$. 
We choose a chain $I_1 = J_0 \prec J_1 \prec \dots \prec J_k = I_1 \lor I_2$.
Because $\rho(I_1 \lor I_2) \ge \rho(I_2) = |I_2| > |I_1| = \rho(I_1)$, by (R2), there exists $i$ such that $\rho(J_{i+1}) - \rho(J_{i}) = 1$.
Let $q \in \coex{J_i}$ such that $J_i \lor q = J_{i+1}$.
By (R2) and (R3), there exists $Z = q \bmod J_i$ and $p \in \coex{I_1}$ with $p \le Z$ such that $\rho(I_1 \lor p) - \rho(I_1) = 1$.
Because $\rho(I_1) = |I_1|$, $\rho(I_1 \lor p) = |I_1| + 1 = |I_1 \lor I_1|$. 
This means that $I_1 \lor p \in \mathcal{I}$.
\end{proof}

\begin{proof}[Proof of Theorem~\ref{thm:sanotype-rank-to-ind}]
Note that (R1) and (R2) is equivalent to (I1) by the same proof as Theorem~\ref{thm:ind-rank}.
In addition, we can easily see that (R3) is equivalent to (I2).
\end{proof}

\subsection{Proofs for Dependence Axiom}

\begin{proof}[Proofs of Theorem~\ref{thm:mod-dep-ind} and Theorem~\ref{thm:llm-dep-ind}]
By Proposition~\ref{prop:localaug}, Theorem~\ref{thm:mod-dep-ind} implies Theorem~\ref{thm:llm-dep-ind}.
Note that the proof of Proposition~\ref{prop:localaug} is valid for a supermatroids on a lower locally modular lattice.
Therefore, we prove Theorem~\ref{thm:mod-dep-ind} in the following.

\noindent ($\Leftarrow$)\
(D1) is trivial.
We prove (D2).
Suppose that $Z \in \ind$ and $D_1 \land D_2 \in \ind$.
By the lower semimodularity, we have $D_1 \land D_2 \prec D_1$.
Thus, $|D_1 \land D_2| < |D_1| = |Z|$. 
By applying augmentation (I2) to the pair $(D_1 \land D_2, Z)$, we obtain $I \in \ind$ with $D_1 \land D_2 < I \le (D_1 \land D_2) \lor Z = D_1 \lor D_2$.
Here, the last equality follows from $D_1 \land D_2 \not \le Z$ and $Z < (D_1 \land D_2) \lor Z \le D_1 \lor D_2$. 

We next prove (D3).
Suppose to the contrary that $X, W \in \ind$.
Then, $|X|+1 \le |W|$.
By applying (I2) to $X$ and $W$, we have $I \in \ind$ with $X < I \le X \lor W$.
This implies that $Y \in \mathcal{I}$; however this is a contradiction.

(2)
(I1) is trivial.
It suffices to show (I2l) because of Proposition~\ref{prop:localaug}.
In other words, it suffices to show augmentation (I2) for a restricted lattice $\lat' = \{X \in \lat \mid I_1 \land I_2 \le X \le I_1 \lor I_2 \}$, i.e., all modular lattices with a  height of  at most 2.
Let $I_1, I_2 \in \ind$ with $|I_1| + 1 = |I_2|$ and $|I_2| \prec \top = I_1 \lor I_2$.
The case $I_1 \le I_2$ is trivial.
Suppose that $I_1 \not \le I_2$.
Then, one of the following holds:
(a) There uniquely exists $J$ such that $I_1 < J < \top$ 
(b) there are two or more $J$ such that $I_1 \prec J \prec \top$.
In case (a), $I_1$ has an augmentation $J$ owing to (D3).
Here, we take $X := I_1$, $W := I_2$, and $Y := J$ in (D3).
In case (b), $X$ has an augmentation owing to (D2).
Here, we take $Z : = I_1$ and $D_i = J_i$ $(i = 1,2)$, where $J_i$ satisfies $I_1 \prec J_i \prec \top$.
\end{proof}

\subsection{Proofs for Dual Matroids}

\begin{proof}[Proof of Proposition~\ref{prop:dualrank}]
The rank function is characterized as 
$r(X) = \max_{B \in \base} |B \land X|$.
Therefore, for the dual matroid, we have
$r^*(\bar{X}) = \max_{B \in \base} |\bar{B} \land^* \bar{X}|^*$,
where we put a superscript $*$ in order to distinguish the $\land$ operation and height of $\lat^*$ from those of $\lat$.
Because $\overline{Y \lor Z} = \bar{Y} \land^* \bar{Z}$ and $|\bar{Y}|^* = |\top| - |Y|$, we have
\begin{align}
    r^*(\bar{X}) = \max_{B \in \base}\left[ |\top| - |\bar{B} \lor \bar{X}|\right] = |\top| - \min_{B \in \base}|B \lor X|.
\end{align}
By the modularity of the height,
$|B \lor X| = |B| + |X| - |B \land X|$.
Hence,
$\min_{B \in \base} |B \lor X| = |B| + |X| - \max_{B \in \base}|B \land X| = r(\top) + |X| - r(X)$.
This completes the proof.
\end{proof}

\subsection{Proofs for Strong Exchange Property and its Applications}
\subsubsection{Birkhoff-type Representation of Modular Lattices}
To prove Theorem~\ref{thm:strongexchange}, we use the Birkhoff-type representation theorem~\cite{Herrmann1994geometric} for modular lattices.
Hence, we summarize the Birkhoff representation.

Let $P$ be a poset and $C$ be a ternary relation on $P$.
We suppose that $C$ is symmetric, that is, $C(q,r, p)$ and $C(r,p,q)$ hold if $C(p,q,r)$ holds for all $p,q,r \in P$.
The symmetric ternary relation $C$ is called a \emph{collinearity relation} if the following condition holds:
All triplets $p,q,r \in P$ with $C(p,q,r)$ are (1) pairwise incomparable and (2) if $p,q, \le s$, then $s \le r$ for all $s \in P$.
An \emph{ordered space} is a poset equipped with a collinearity relation.
A \emph{subspace} of $P$ is an ideal $I$ satisfying the following property:
If $C(p,q,r)$ and $p,q \in I$, then $r \in I$.
A \emph{projective ordered space} is an ordered space satisfying certain axioms~\cite{Herrmann1994geometric}.

Herrmann, Pickering, and Roddy~\cite{Herrmann1994geometric} proposed a Birkhoff-type representation theorem of a nodular lattice.
We can construct a modular lattice $\lat(P)$ from a projective ordered space $P$ as follows.
Let $\lat(P)$ be the family of the subspaces of $P$ and consider the inclusive order.
Conversely, we can construct a projective ordered space $P(\lat)$ from a modular lattice $\lat$ as follows.
Let $P(\lat)$ be the family of join-irreducible elements of $\lat$.
Consider the order on $P(\lat)$ induced by $\lat$.
In addition, $p,q,r \in P(\lat)$ is $C(p,q,r)$ if and only if (1) $p,q,r$ are pairwise incomparable and $p \lor q = q \lor r = r \lor p$.
Then, the following representation theorem holds.
\begin{theorem}[\cite{Herrmann1994geometric}]
Let $\lat$ be a modular lattice and $P$ be a projective ordered space.
Then, $\lat(P)$ is a modular lattice and $P(\lat)$ is a projective ordered space.
Also, $\lat \cong \lat(P(\lat))$ and $P \cong P(\lat(P))$.
\end{theorem}
In the rest of this section, we use this identification between a modular lattice $\lat$ and a projective ordered space $P$.
The notation  $p \in X$ means that $p \le X$ for $p \in P(\lat)$ and $X \in \lat$.
Recall that the join-irreducible elements of $\lat$ are denoted by lower-case letters and the others are denoted by upper-case letters.

An important property of projective ordered spaces is the following.

\begin{definition}[Regularity Axiom~\cite{Herrmann1994geometric}]
\label{def:regularity-axiom}
Let $P$ be a projective ordered space and $p,q,r \in P$ such that $C(p,q,r)$.
For all $\underline r \in P$ such that $\underline r \le r$, $\underline r \not \le p$, and $\underline r \not \le q$, there exist $\underline p, \underline q \in P$ such that $\underline p \le p$, $\underline q \le q$, and $C(\underline p, \underline q, \underline r)$.
\end{definition}

\begin{lemma}[\cite{Herrmann1994geometric}]
\label{lem:join-projective-ordered-space}
Let $P$ be a projective ordered space and $S,T \in \lat(P)$ be the subspaces.
Then,
\begin{align}
    S \lor T = S \cup T \cup \{r \in P \mid \text{there exist $s \in S$ and $t \in T$ such that $C(s,t,r)$}\}.
\end{align}
\end{lemma}

Based on these property, we can prove a useful lemma.
\begin{lemma}
\label{lem:strengthend-downward-DR}
Let $\lat$ be a modular lattice and $f \colon \lat \to \mathbb{R}$ be a downward DR-submodular function.
For all $X, Y \in \lat$ and $x \in \adm{X} \cap \adm{X \lor Y}$, there exists $x' \in \adm{X}$ such that $x = x' \bmod X$, and for all $\underline x' \le x'$ with $\underline x' \in \adm{Y}$, we have the following:
\begin{align}
    \label{eq:strengthened-downward-dR}
    f((X \lor Y) \lor x) - f(X \lor Y) \le f(Y \lor \underline x') - f(Y).
\end{align}
\end{lemma}
\begin{proof}
By the downward DR-submodularity, there exists $w = x \bmod X \lor Y$ such that
\begin{align}
    f((X \lor Y) \lor w) - f(X \lor Y) \le f(Y \lor \underline w) - f(Y),
\end{align}
for all $\underline w \le w$ with $\underline w \in \adm{Y}$.
By Lemma~\ref{lem:join-projective-ordered-space} and $w \in (X \lor Y) \lor x = (X \lor x) \lor Y$, at least one of the following holds:
\begin{itemize}
    \item[(1)] $w \in X \lor x$;
    \item[(2)] $w \in Y$;
    \item[(3)] There exists $x_w \in X \lor x$ and $y_w \in Y$ such that $C(x_w, y_w, w)$. 
\end{itemize}
In the case (1), $x = w \bmod X \lor Y$ implies that $w \not \in X$.
Therefore, $w = x \bmod X$ and we can take $x'$ as $x' := w$.
Case (2) contradicts to the fact $x = w \bmod X \lor Y$.
In case (3), we can take $x'$ as $x' = x_w$.
Indeed, $x = w \bmod X \lor Y$ implies that $x_w \not \in X$.
In addition, let $\underline x' \le x'$ such that $\underline x' \in \adm{Y}$.
If $\underline x' \le w$, then (\ref{eq:strengthened-downward-dR}) holds by the choice of $w$.
Otherwise, $\underline x' \not \le y_w$ because $\underline x' \in \adm{Y}$.
Therefore, the regularity axiom (Definition~\ref{def:regularity-axiom}) implies that there exist $\underline w \le w$ and $\underline y \le y_w$ such that $C(\underline x', \underline w, \underline y)$.
In particular, $\underline x' = \underline w \bmod Y$.
Therefore, (\ref{eq:strengthened-downward-dR}) holds by the choice of $w$.
\end{proof}

\subsubsection{Proofs}

\begin{proof}[Proof of Theorem~\ref{thm:strongexchange}]

We first construct $H \ge \mathring{X}$ as follows.
Set $H \gets \mathring{X}$.
If there exists $y \le Y$ such that $y \in \adm{H}$ and $r(H \lor y) = r(H)$, set $H \gets H \lor y$.
Continuing this process, the enlargement stops at some steps and we obtain a maximal $H$.
This $H$ has the following properties.
First, $H = \mathring{X} \lor (Y \land H)$. Indeed, by construction, $H = \mathring{X} \lor y_1 \lor y_2 \lor \dots \lor y_l$.
Here, $y_i$ is the $i$-th $y$ in the construction of $H$.
Hence, we have $H = \mathring{X} \lor (Y \land H)$.
Second, for all $y \in \adm{H}$ with $y \in Y$, we have $r(H \lor y) = r(H) + 1$.
Thirdly, $x \in \adm{H}$ and $r(H \lor x) = r(H) + 1$ for all $x \in \adm{X}$ with $\mathring{X} \lor x = X$. 
Indeed, if $x \le H$, then $H \ge \mathring{X} \lor x = X$ and $r(X) = r(\mathring X) + 1$.

We next take $w$.
Let $x \le X$ be some join-irreducible element with $\mathring{X} \lor x = X$.
Because $r(H \lor x) - r(H) = 1$ and $H = \mathring X \lor (Y \land H)$, 
Lemma~\ref{lem:strengthend-downward-DR} implies that there exist $w \in \adm{H}$ such that $w = x \bmod \mathring X$ and $r((Y \land H) \lor \underline w) - r(Y \land H) = 1$ for all $\underline{w} \le w$ with $\underline{w} \in \adm{Y \land H}$.
Because $Y \land H \in \ind$, we have $(Y \land H) \lor \underline w \in \ind $  for all $\underline{w} \le w'$ with $\underline{w} \in \adm{Y \land H}$.

We then take $Y'$. 
Because $((Y \land H) \lor \underline w) \lor Y = Y \lor \underline{w}$ for any $\underline{w}$ in the statement of this theorem, (B2) implies the existence of $Y' \in \base$ such that $(Y \land H) \lor \underline{w} \le Y' \le  Y \lor \underline{w}$.
Note that $Y' \ge (Y \land H) \lor \underline{w} \ge (Y \land X) \lor \underline{w}$.

We finally take $\underline y$ and $v$.
Let $\underline y$ be any $\underline y \le y$ such that $\underline y \in \adm{H}$.
Because $r(H \lor \underline y) - r(H) = 1$ and $H = (Y \land H) \lor \mathring X$, Lemma~\ref{lem:strengthend-downward-DR} implies that there exists $v = y \bmod Y \land H$ such that $r(\mathring X \lor \underline v) - r(\mathring X) = 1$ for all $\underline v \le v$ with $\underline v \in \adm{\mathring X}$.
Because $Y \land H \le Y \land Y'$, we have $v = y \bmod Y \land Y'$.
Because $\mathring X \in \ind$ and $\mathring X \prec X \in \base$, we have $\mathring X \lor \underline v \in \ind$ for all $\underline v \le v$ with $\underline v \in \adm{\mathring X}$.
Because $v \not \in H$ and $\mathring X \le H$, the set $\{\underline v \le \adm{\mathring X} \mid \underline v \le v\}$ is non-empty.

\end{proof}

\begin{proof}[Proof of Proposition~\ref{prop:strongex-to-base}]
Let $X \le Y \in \lat$ and $B_1, B_2 \in \base$ with $B_1 \ge X$ and $B_2 \le Y$.
We prove (B2) by the induction on $k := |B_1| - |B_1 \land Y|$.
If $k=0$, then $B_1 \le Y$ and the statement trivially holds.
Consider the case of $k > 0$.
Apply Theorem~\ref{thm:strongexchange} to $B_1, B_2$, and arbitrary $\mathring{X}$ with $B_1 \land Y \le \mathring{X} \prec B_1$.
Here, $\mathring X \ge X$.
Then, we obtain $B_1'  \in \base$ such that $|B_1'| - |B_1' \land Y| = k - 1$ and $B_1' \ge X$.
By applying an induction hypothesis to $X \le B_1'$ and $Y \ge B_2$, we have the following statement.
\end{proof}

\begin{proof}[Proof of Theorem~\ref{thm:greedy-valuated}]
We prove that, for each $t$, the solution $X_t$ maximizes $f$ among all elements in $\lat$ whose height is $t$ by the induction on $t$.
The case $t = 0$ is trivial.
For $t \ge 1$, let $O_t$ be the optimal solution of height $t$ such that $|O_t \land X_t|$ is maximum.
If $X_t \neq O_t$, then
\begin{align}
    w(X_t) + w(O_t) 
   & = w(X_{t-1} \lor a_t) + w(O_t) \\
    &\le w(X_{t-1} \lor b) + w(O_t'),
\end{align}
where we obtain $X_{t-1} \lor b$ and $O_t'$  by applying the definition of the valuated supermatroid for $(X_t, X_{t-1}, O_t)$. 
Because the greedy property of the algorithm implies that $w(X_{t-1} \lor a_t) \ge w(X_{t-1} \lor b)$, we have  $O_t'$ as an optimal solution.
However, $|X_t \land O_t'| > |X_t \land O_t|$, it contradicts the choice of $O_t$.
Therefore, we have $X_t = O_t$.
\end{proof}

\begin{proof}[Proof of Proposition~\ref{prop:pca-is-valuated-supermatroid}]
For simplicity, we denote $\omega_{\mathrm{PCA}}$ by $\omega$ in this proof.
If $|X| - |X \land Y| = 0, 1$, the function $\omega$ trivially satisfies the definition of a valuated supermatroid.
Let us consider the case of $|X| - |X \land Y| \ge 2$.
Let $d' = (\Pi_X - \Pi_{X \land Y}) d$.
If $d' \not \in \mathring X$, then $\omega(\mathring X) = \omega(X)$.
Take any $w$ such that $\mathring X \lor w = X$. 
In addition, any $Y'$ such that $\omega(Y') \ge \omega(Y)$ and $(X \land Y) \lor w \le Y' \prec Y \lor w$, i.e.,
take $Y'$ such that $d'' \in Y'$, where $d'' = (\Pi_Y - \Pi_{X \land Y}) d$.
Such $Y'$ exists because $|Y| - |X \land Y| \ge 2$.
We then have $\omega(X) + \omega(Y) \le \omega(\mathring{X} \lor v) + \omega(Y')$ by taking an arbitrary $v$ such that $Y' \lor v = Y \lor w$.
Therefore, $\omega$ is a valuated supermatroid.
If $d \in \mathring X$, take $w = \mathrm{span}(d')$.
Then, $\mathring X \lor w = X$. 
Let  $v = \mathrm{span}(d'')$.
Take any $Y'$ with $(X \land Y) \lor w \le Y' \prec Y \lor w$ and $Y' \lor v = Y \lor w$.
Then, $\omega(Y') \ge \omega((X \land Y) \lor w) = \omega(X)$ because $d'$ is orthogonal to $X \land Y$. 
In addition, we have $Y' \lor v = Y \lor w$ and $\omega(\mathring X \lor v) \ge \omega((X \land Y) \lor v) = \omega(Y)$ because $v$ is orthogonal to $X \land Y$.
In addition, because $\omega(X) + \omega(Y) \le \omega(\mathring{X} \lor v) + \omega(Y')$, we have proven that $\omega$ is a valuated supermatroid.
\end{proof}

To prove Theorem~\ref{thm:greedy-DR-matroid}, we need the following lemmas.
\begin{lemma}
\label{lem:diamond-matching}
Suppose that $\lat$ is an atomic modular lattice.
Let $X, O \in \base$.
In addition, let $X \land O = Z_0 \prec Z_1 \prec \dots \prec Z_\alpha = X$ be any chain.
Then, we can take a chain $X \land O = Y_0 \prec Y_1 \prec \dots \prec Y_\alpha = O$ that satisfies the following condition:
For all $i =1, 2, \dots, \alpha$, there exists $y_i$ such that $Y_i = Y_{i-1} \lor y_i$ and $Z_{i-1} \lor y_i \in \ind$.
\end{lemma}
\begin{proof}
We prove this lemma by the induction on $\alpha = |X| - |X \land O|$.
If $\alpha = 0$, then $X \land O = X$ and the lemma trivially holds.
We consider the case $\alpha > 0$.
Let $Y_\alpha = O$.
By applying a strong basis exchange property (Corollary~\ref{cor:strongexchange}) to the pair $(X, O)$ and $Z_{\alpha -1} \prec X$, we obtain the following:
$x$, such that $Z_{\alpha -1} \lor x = X$;
$O' \prec O \lor x$, such that $O' \in \base$ and $O' \ge (X \land O) \lor x$;
and $o \in O$, such that $O' \lor o = O \lor x$ and $Z_{\alpha-1} \lor o \in \base$.
Let $Y_{\alpha -1} = O' \land O$ and $y_\alpha = o$.
The modular identity implies that $Y_{\alpha-1} \lor y_\alpha = (O' \land O) \lor o = (O' \lor o) \land O = O$. 
By applying the induction hypothesis on  $Y_{\alpha-1}$, the truncated chain $Z_0 \prec Z_1 \prec \dots \prec Z_{\alpha-1}$, and the truncated modular supermatroid on $\{ l \in \lat \mid |l| \le k-1\}$, we can take a chain $Y_0 \prec Y_1 \prec \dots \prec Y_{\alpha-1}$ satisfying the condition of this Lemma.
Then, $Y_0 \prec Y_1 \prec \dots \prec Y_{\alpha-1} \prec Y_{\alpha}$ is the desired chain.
\end{proof}

\begin{lemma}
\label{lem:matchin-lower}
Suppose that $\lat$ is an atomic modular lattice.
Let $X, O \in \base$.
Let $\perp = X_0 \prec X_1 \prec \dots \prec X_k = X$ be any chain.
Then, we can take a chain $X \land O = Y_0 \prec Y_1 \prec \dots \prec Y_\alpha = O$ that satisfies the following condition:
For all $i =1, 2, \dots, \alpha$, there exists $y_i$ and $j_i \in [k]$ such that $Y_i = Y_{i-1} \lor y_i$ and $X_{j_i} \lor y_i \in \ind$.
Furthermore, $j_i$ is injective with respect to $i$.
\end{lemma}
\begin{proof}
For any $i \in [k]$, one of the following holds: $X_i \land O = X_{i-1} \land O$ or $X_i \land O \succ X_{i-1} \land O$.
If the latter case holds, then we can take $o_i \in O$ such that $X_i = X_{i-1} \lor o_i$.
Otherwise, we take $x_i \in X \setminus O$ such that $X_i = X_{i-1} \lor x_i$.
Then, we have $X = (o_{m_1} \lor o_{m_1} \lor \dots \lor o_{m_\beta}) \lor (x_{n_1} \lor x_{n_2} \lor \dots \lor x_{n_\alpha})$.
Notice that $X \land O = (o_{m_1} \lor o_{m_1} \lor \dots \lor o_{m_\beta})$ and $\beta = k - \alpha$.
By using
$\{x_{n_i}\}_{i=1,2,\dots,\alpha}$, we construct a chain $X \land O = Z_0 \prec Z_1 \prec \dots Z_\alpha = X$ as $Z_i = (X \land O) \lor x_{n_1} \lor x_{n_2} \lor \dots \lor x_{n_i}$.
By using Lemma \ref{lem:diamond-matching} for $\{Z_i\}_{i=1,2,\dot,\alpha}$, we can take a desired chain $X \land O = Y_0 \prec Y_1 \prec \dots \prec Y_k = O$.
Indeed, we can take $j_i = n_i - 1$ because $X_{n_\alpha - 1} \lor y_i \le Z_{\alpha -1 } \lor y_i \in \ind$.
This $j_i$ is injective with respect to $i$.
\end{proof}

\begin{lemma}
\label{lem:matching}
Suppose that $\lat$ is an atomic modular lattice.
Let $X, O \in \base$.
Let $\perp = X_0 \prec X_1 \prec \dots \prec X_k = X$ be any chain.
Then, we can take a chain $X = W_0 \prec W_1 \prec \dots \prec W_\alpha = X \lor O$ that satisfies the following conditions:
(1) We can take $w_i$ such that $W_i = W_{i-1} \lor w_i$.
(2) For all $i= 1,2,\dots, \alpha$, there exists $j_i$ such that $X_{j_i} \lor w_i \in \ind$ and $j_i$ is injective with respect to $i$.
\end{lemma}
\begin{proof}
Obtain $\{Y_i\}$, $y_i$, and $j_i$ $(i=1,2,\dots,\alpha)$ by applying Lemma~\ref{lem:matchin-lower} to $\{X_i\}_{i=0,1,\dots,k}$ and $O$.
Let $W_i = X \lor Y_i$.
It suffices to show that $w_i = y_i$ satisfies the condition $W_i = W_{i-1} \lor y_i$.
Notice that
$X \lor y_1 \lor y_2 \lor \dots \lor y_\alpha = X \lor ((X \land O) \lor y_1 \lor y_2 \lor \dots \lor y_\alpha)) = X \lor O$.
In addition, $W_i =  X \lor ((X \land O) \lor y_1 \lor y_2 \lor \dots \lor y_i))$.
Because the modularity of the height implies that $\alpha = |X \lor O| - |X|$, 
we have $|W_i| - |W_{i-1}| = 1$ for all $i = 1,2,\dots, \alpha$.
This ensures that $W_i = W_{i-1} \lor y_i$ because $W_i \ge W_{i-1} \lor y_i$.
\end{proof}

\begin{proof}[Proof of Theorem~\ref{thm:greedy-DR-matroid}]
Let $X$ be the output of the algorithm and $O$ be the optimal solution.
By Lemma \ref{lem:matching}, we can take a chain $X = W_0 \prec W_1 \prec \dots \prec W_\alpha = X \lor O$.
Then,
\begin{align}
    f(X \lor O) - f(X) 
    = \sum_{i=1}^\alpha (f(W_{i+1}) - f(W_i))
    = \sum_{i = 1}^\alpha (f(W_i \lor w_i) - f(W_i)) =: (*).
\end{align}
Here, we take $w_i$ as in Lemma \ref{lem:matching}.
By the second condition of Lemma \ref{lem:matching}, the strong DR-submodualrity, and the greedy property, we have
\begin{align}
    (*) 
    \le \sum_{i=1}^\alpha f(X_{j_i - 1} \lor w_i) - f(X_{j_i -1})
    &\le \sum_{i=1}^\alpha f(X_{j_i}) - f(X_{j_i - 1})\\
    &\le \sum_{i=1}^k f(X_{i}) - f(X_{i-1})
    \le f(X),
\end{align}
where $j_i$ is defined as in Lemma \ref{lem:matching} and we used the fact that $j_i$ is injective with respect to $i$.
\end{proof}

\begin{remark}
\label{rem:downward-max-generalization-hard}
A generalization of Theorem~\ref{thm:greedy-DR-matroid} to general modular lattices is difficult because we cannot prove Lemma~\ref{lem:diamond-matching}, which was a key to proving Theorem~\ref{thm:greedy-DR-matroid}.
This difference comes from that of the strong exchange property.
In atomic modular lattices, we have $Y' \lor v = Y \lor w$, where $w$ and $v$ are the ``exchanged elements'' (Corollary~\ref{cor:strongexchange}).
However, in general modular lattices, the exchanged element $v$ might be in $Y'$ and $Y' \lor v = Y' < Y \lor \underline w$.
In such cases, we cannot guarantee the property $Y_i = Y_{i-1} \lor y_i$ in the proof of Lemma~\ref{lem:diamond-matching}.

In addition, a generalization of Theorem~\ref{thm:greedy-DR-matroid} to downward DR-submodular functions is difficult to achieve.
In the proof of Theorem~\ref{thm:greedy-DR-matroid}, we used the strong DR-submodularity and obtained the following:
\begin{align}
    f(W_i \lor w_i) - f(W_i) \le f(X_{j_i - 1} \lor w_i) - f(X_{j_i-1}).
\end{align}
If we use the downward DR-submodularity instead of the strong DR-submodularity, we obtaine
\begin{align}
    f(W_i \lor w_i) - f(W_i) \le f(X_{j_i - 1} \lor v_i) - f(X_{j_i-1}),
\end{align}
for some $v_i = w_i \bmod W_i$.
However, such $v_i$ might not satisfy $X_{j_i} \lor v_i \in \ind$, which is necessary for the next boundby the greedy property.
\end{remark}

\begin{proof}[Proof of Theorem~\ref{thm:greedy-DR-matroid-curvature}]
We prove this theorem through an induction on $k$.
The case $k = 1$ is trivial.
Consider  $k > 1$.
Let $X_{i}$ be the output of the greedy algorithm after the $i$-th iteration, and $a_i$ be the chosen element, that is,  $X_i = X_{i-1} \lor a_i$. 
Let $X^*$ be the optimal solution.
By (I2), there exists $w \le X^*$ such that $X_{i-1} \lor w$ is feasible.
By greedy property and curvature, we have the following:
\begin{align}
    f(X_k) - f(X_{k-1}) &\ge f(X_{k-1} \lor w) - f(X_{k-1})\\
    &\ge (1-c) f(\underline w),
\end{align}
where $\underline{w} \le w$ is a minimal element.
By upward DR-submodularity,
there exists $\mathring{X^*} \prec X^*$ such that
\begin{align}
    f(\underline w) = f(\underline w) - f(\bot)
    \ge f(X^*) - f(\mathring{X^*}). 
\end{align}
By the induction hypothesis, $(1-c)f(\mathring{ X^*}) \le f(X_{k-1})$.
Therefore,
\begin{align}
    f(X_k) - f(X_{k-1}) \ge (1-c)f(X^*) - f(X_{k-1}).
\end{align}
Hence, $X_k$ is a (1-c)-approximation solution.
\end{proof}

\section*{Acknowledgment}

We thank Hiroshi Hirai for fruitful discussion.
This work was supported by JSPS KAKENHI Grant Number 19K20219.
The second author is financially supported by JSPS Research Fellowship Grant Number JP19J22607.


\bibliographystyle{plain}
\bibliography{main}

\end{document}